\documentclass{amsart}

\usepackage{amsmath,amssymb,amsfonts}
\usepackage{graphicx}
\usepackage{tikz}
\usepackage{esint}

\usepackage[sort,nospace]{cite}

\usepackage[colorlinks=true, pdfstartview=FitV, linkcolor=blue,citecolor=blue, urlcolor=blue]{hyperref}
\usepackage[capitalize,nameinlink,noabbrev]{cleveref}

\title[Energy dissipation near the outflow boundary]{Energy dissipation near the outflow boundary in~the vanishing viscosity limit}

\author[J. Yang]{Jincheng Yang}
\address{School of Mathematics, Institute for Advanced Study, 1 Einstein Dr, Princeton, NJ 08540, USA}
\email{jcyang@ias.edu}

\author[V.R. Martinez]{Vincent R. Martinez}
\address{Department of Mathematics and Statistics, CUNY Hunter College, 695 Park Ave, New York City, NY 10065, USA}
\address{
Department of Mathematics, CUNY Graduate Center, 365 5th Ave, New York City, NY 10016, USA
}
\email{vrmartinez@hunter.cuny.edu}

\author[A.L. Mazzucato]{Anna L. Mazzucato}
\address{Department of Mathematics, Pennsylvania State University, University Park, PA 16802, USA}
\email{alm24@psu.edu}

\author[A.F. Vasseur]{Alexis F. Vasseur}
\address{Department of Mathematics, The University of Texas at Austin, 2515 Speedway, Austin, TX 78712, USA}
\email{vasseur@math.utexas.edu}

\date{\today}

\thanks{\textit{Acknowledgment}. The authors would like to thank American Institute of Mathematics (AIM) for hosting the workshop ``Small scale dynamics in incompressible fluid flows'', where this project \cite{AIM2023} was initiated. AIM receives major funding from the US National Science Foundation and the Fry Foundation. The first author was partially supported by the US National Science Foundation under Grant No.~DMS-1926686. The second author was partially supported by the US National Science Foundation under Grant No.~DMS-2213363 and DMS-2206491, as well as the Dolciani Halloran Foundation. The third author was partially supported by the US National Science Foundation under Grant No.~DMS-2206453 and by Simons Foundation under Grant No.~1036502. The fourth author was partially supported by the US National Science Foundation under Grant No.~DMS-2219434 and DMS-2306852.}

\newtheorem{theorem}{Theorem}[section]
\newtheorem{lemma}[theorem]{Lemma}
\newtheorem{proposition}[theorem]{Proposition}

\theoremstyle{definition}

\theoremstyle{remark}
\newtheorem{remark}[theorem]{Remark}

\renewcommand*{\d}{\mathop{\kern0pt\text{\upshape d}}\!{}}
\let\div\relax
\DeclareMathOperator{\div}{div}
\DeclareMathOperator{\tr}{tr}

\newcommand{\bu}{\boldsymbol u}
\newcommand{\bU}{\boldsymbol U}
\newcommand{\bV}{\boldsymbol V}
\newcommand{\bn}{\boldsymbol n}
\newcommand{\be}{\boldsymbol e}
\newcommand{\bx}{\boldsymbol x}
\newcommand{\by}{\boldsymbol y}
\newcommand{\bv}{\boldsymbol v}
\newcommand{\bF}{\boldsymbol F}

\newcommand{\bA}{\boldsymbol A}
\newcommand{\bB}{\boldsymbol B}

\newcommand{\tE}{\mathcal{E} _{\bU}}
\newcommand{\tomega}{\widetilde {\boldsymbol \omega}}
\newcommand{\normal}{^{\boldsymbol n}}
\newcommand{\tangential}{^{\boldsymbol \tau}}
\newcommand{\uE}{\bu _\mathrm E}
\newcommand{\uBL}{\bu _{\mathrm{BL}} ^\nu}
\newcommand{\unu}{\bu _\nu}
\newcommand{\vnu}{\bv _\nu}
\newcommand{\pE}{p _\mathrm E}
\newcommand{\pnu}{p _\nu}
\newcommand{\cO}{\mathcal O}
\newcommand{\Ub}{\bar U}
\newcommand{\Vb}{\bar V}
\newcommand{\Gammanuu}{\Gamma ^- _{\nu / \Ub}}

\newcommand{\bzero}{\boldsymbol 0}

\newcommand{\Tone}{T _1}
\newcommand{\Ttwo}{T _2}
\newcommand{\KE}{\mathcal{E}}

\begin{document}

\begin{abstract}
    We consider the incompressible Navier--Stokes and Euler equations in a bounded domain with non-characteristic boundary condition, and study the energy dissipation near the outflow boundary in the zero-viscosity limit. We show that in a general setting, the energy dissipation rate is proportional to $\bar U \bar V ^2$, where $\bar U$ is the strength of the suction and $\bar V$ is the tangential component of the difference between the Euler and the Navier--Stokes solutions on the outflow boundary. Moreover, we show that the enstrophy within a layer of order $\nu / \bar U$ is comparable with the total enstrophy. The rate of enstrophy production near the boundary is inversely proportional to $\nu$.
\end{abstract}

\maketitle
\tableofcontents

\section{Introduction}

We study the rate of energy dissipation of an incompressible fluid under the inflow/outflow boundary condition in the limit of vanishing viscosity. In most of the literature studying the Euler system or the Navier--Stokes system, impermeability is imposed on the boundary, i.e., characteristic boundary conditions. This forbids the fluid to enter or leave the domain. However, for many physical applications, such as drag reduction, it is important to consider injection on one portion of the boundary and suction on another. For instance, one may investigate the flow in a channel with injection on the top wall and suction on the bottom wall, as shown in \cref{fig:channel}. One can distinguish the fluid behavior near the boundary between characteristic and non-characteristic boundary conditions. In this paper, we study the energy dissipation and enstrophy production near the outflow boundary in presence of non-characteristic boundary conditions, in the vanishing viscosity limit, over general domains.

Let $\Omega \subset \mathbb R ^d$ be an open and bounded domain or a finite periodic channel, with $d = 2, 3$. Suppose $\partial \Omega$ is smooth and can be decomposed into two disconnected components:
\begin{align*}
    \partial \Omega = \Gamma = \Gamma ^+ \cup \Gamma ^-, \qquad \Gamma ^+ \cap \Gamma ^- = \varnothing.
\end{align*}
One can include an impermeable component $\Gamma^0$, disjoint from $\Gamma^{\pm}$. For simplicity, we take $\Gamma^0 = \varnothing$.
We consider an injection of incompressible fluid on $\Gamma ^+$ and a suction on $\Gamma ^-$ and denote its velocity by $\bU: (0, \infty) \times \partial \Omega \to \mathbb R ^d$. That is, $\bU \cdot \bn < 0$ on $\Gamma ^+$ and $\bU \cdot \bn > 0$ on $\Gamma ^-$, where $\bn$ is the unit outer normal vector. Due to incompressibility, $\bU$ needs to satisfy the flux compatibility condition at every $t > 0$:
\begin{align}
    \label{eqn:compatibility}
    \int _{\partial \Omega} \bU(t) \cdot \bn \d \sigma = \int _{\Gamma ^+} \bU (t) \cdot \bn \d \sigma + \int _{\Gamma ^-} \bU(t) \cdot \bn \d \sigma = 0.
\end{align}
We denote the normal component of $\bU$ by $\bU \normal = U \normal \bn = (\bU \cdot \bn) \bn$ and the tangential part by $\bU \tangential = \bU - \bU \normal$. The compatibility condition can thus be written as $\int _{\partial \Omega} U \normal \d \sigma = 0$, while the inflow/outflow conditions can be written as $U \normal<0$ on $\Gamma^+$ and $U \normal>0$ on $\Gamma^-$. Throughout the article, we assume $\bU \in C ^1 ([0, \infty); C ^2(\partial \Omega))$.
We discuss the well-posedness of the Navier--Stokes and Euler equations with these boundary conditions below. Our assumption on $\bU$ is the minimal regularity required for our result to hold, {\em assuming} a classical solution of the Euler equations exists.

\begin{figure}[htbp]
    \centering    
    \begin{tikzpicture}
        \def\width{6}
        \def\height{3}

        \draw (0,\height) -- (\width,\height) node[right] {$\Gamma ^+$};
        \draw[-latex] (\width/2 - 0.3, \height + 1) -- (\width/2 + 0.3, \height) node[midway, anchor=south west] {$\bU$};

        \foreach \x in {0.1,0.3,...,\width}{
            \draw (\x,\height) -- (\x+0.2,\height+0.2);
            \draw (\x,0) -- (\x-0.2,-0.2);
        }

        \draw (\width/2, \height/2) node {$\Omega$};

        \draw (0,0) -- (\width,0) node[right] {$\Gamma ^-$};
        \draw[-latex] (\width/2 - 0.3, 0) -- (\width/2 + 0.3, -1) node[anchor = south west] {$\bU$};

    \end{tikzpicture}
    \caption{Example of a channel with injection on the top wall and suction on the bottom wall.}
    \label{fig:channel}
\end{figure}

We consider a classical Euler solution $\uE: (0, T) \times \Omega \to \mathbb R ^d$, $\pE: (0, T) \times \Omega \to \mathbb R$ with the inflow/outflow boundary condition on some time interval:
\begin{align}
    \label{eqn:euler}
    \begin{cases}
        \partial _t \uE + (\uE \cdot \nabla) \uE + \nabla \pE = \bzero & \text{ in } (0, T) \times \Omega \\
        \div \uE = 0 & \text{ in } (0, T) \times \Omega \\
        \uE \vert _{t = 0} = \uE ^0 & \text{ in } \Omega \\
        \uE \vert _{\Gamma ^+} = \bU & \text{ on } (0, T) \times \Gamma ^+ \\
        \uE \normal \vert _{\Gamma ^-} = \bU \normal & \text{ on } (0, T) \times \Gamma ^- .
    \end{cases}
\end{align}
We suppose that the initial condition $\uE ^0$ is a $C ^1 (\Omega)$, divergence-free vector field. We note that on the inflow boundary $\Gamma ^+$, the velocity field $\uE$ is completely prescribed, whereas on the outflow boundary $\Gamma ^-$ only the normal component of $\uE$ is prescribed. Here we employ the same notation used for $\bU$ to denote the tangential and normal components of $\uE$, $\uE \tangential$ and $\uE\normal$ respectively. For the local-in-time well-posedness results of the Euler system with inflow/outflow boundary condition, we refer to \cite{antontsev1989,gie2022,gie2023,kukavica2023,bravin2022,chemetov08,petcu2006} and  references therein (we also mention the work \cite{bardos1998}, where pressure-vorticity conditions at inflow are considered). We observe that the global-in-time well-posedness for the Euler equations with inflow/outflow boundary data is only known in 2D for an analytic domain and analytic boundary data that become suitably small in an analytic norm as $t \to \infty$ \cite{kukavica2023}. In this work, we only assume the existence of \textit{some} classical solution $\uE$, which is of class $C ^1$ in time and space up to time $T$, without further assumptions on the regularity, stability, or uniqueness.
Also note that although viscous forces are absent in an Euler flow, the kinetic energy need not be conserved in the context of non-characteristic boundary conditions, since the energy flux, $\frac12\int _{\partial \Omega} \lvert \uE \rvert ^2 U \normal \d \sigma$, and the work due to pressure, $\int _{\partial \Omega} \pE U \normal \d \sigma$, are generally nonzero.

In the case of characteristic boundary conditions, it is an open problem to decide whether the Euler system can be considered a good approximation to the Navier--Stokes system in the vanishing viscosity limit. Throughout the paper by characteristic boundary conditions, we mean  no-slip conditions for Navier—Stokes and no-penetration conditions for Euler, excluding the case of Navier slip-with-friction conditions. Let $\{\unu ^0\} _{\nu > 0}$ be a family of divergence-free initial data converging to $\uE ^0$ in $L ^2 (\Omega)$ as $\nu \to 0$, where $\nu \in (0, \infty)$ is the kinematic viscosity coefficient. We consider Leray--Hopf weak solutions $(\unu, \pnu)$ to the following family of Navier--Stokes systems with inflow/outflow boundary condition, indexed by $\nu$:
\begin{align}
    \label{eqn:Navier-Stokes}
    \begin{cases}
        \partial _t \unu + (\unu \cdot \nabla) \unu + \nabla \pnu = \nu \Delta \unu & \text{ in } (0, T) \times \Omega \\
        \div \unu = 0 & \text{ in } (0, T) \times \Omega \\
        \unu \vert _{t = 0} = \unu ^0 & \text{ in } \Omega \\
        \unu \vert _{\partial \Omega} = \bU & \text{ on } (0, T) \times \partial \Omega.
    \end{cases}
\end{align}
The existence of weak solutions is known (see  \cite{alekseenko94,bertagnolio1999,boyer07} and references therein, see also \cite{chemetov2014}).
In contrast to the Euler system \eqref{eqn:euler}, the boundary value of $\unu$ is completely prescribed on both the inflow and the outflow boundary.
In this case, it is also possible to consider the situation where the domain is not smooth and the inflow/outflow/characteristic parts of the boundary, $\Gamma^{\pm,0}$, touch \cite{zajaczkowski87.2,Zajaczkowski09,Zajaczkowski87.1}.

In the case of no-slip boundary conditions, $\bU = \bzero$, the problem of whether the Navier--Stokes system \eqref{eqn:Navier-Stokes} converges to the Euler system \eqref{eqn:euler} with impermeable boundary condition, $\bU \normal=\bzero$, as $\nu \to 0$ is often referred to as the \textit{inviscid} or {\em zero viscosity limit problem}, which has been a long-standing open problem. 
One of the main analytic challenges in studying the inviscid limit arises from the disparity in the boundary conditions between the Navier--Stokes equations and the Euler equations. In the celebrated seminal work of \cite{prandtl1904}, Prandtl used a singular asymptotic expansion to show that at least in the laminar flow regime, the discrepancy between the Euler solution from the corresponding Navier--Stokes solution is predominantly captured by the behavior of the Euler solution in a boundary layer, the thickness of which is of order $\nu ^{1/2}$. In the regime of turbulent flows, however, boundary layers typically separate; mathematically, it is generally accepted that layer separation corresponds to the formation of singularities in the boundary layer equations, indicating a fundamental limitation in the Prandtl approximation. Moreover, one generally does not have effective control over the production of vorticity on the boundary, which is crucial to proving convergence.  Nevertheless, in \cite{kato1984} Kato established a characterization of the inviscid limit problem. As in Kato's work and most other known criteria for the validity of the inviscid limit, in this work we assume $\unu$ is a family of Leray--Hopf weak solutions to the Navier--Stokes equations and $\uE$ is a solution to the Euler equations with sufficient regularity. Kato showed that $\unu$ converges strongly to $\uE$ in $L ^\infty (0, T; L ^2 (\Omega))$ if and only if the following condition is satisfied:
\begin{align}
    \label{eqn:kato}
    \lim _{\nu \to 0} \int _0 ^T \int _{\Gamma _{c \nu}} \nu \lvert \nabla \unu \rvert ^2 \d x \d t = 0,
\end{align}
for some positive constant $c$, where $\Gamma _\delta$ denotes an interior tubular neighborhood of the boundary of positive width $\delta$ defined as
\begin{align*}
    \Gamma _\delta := \left\lbrace 
        x \in \Omega: \operatorname{dist} (x, \partial \Omega) < \delta
    \right\rbrace.
\end{align*}
We remark that the thickness of the Kato boundary layer is significantly smaller  than that of the Prandtl boundary layer. Equivalent criteria to Kato's have been derived in the literature \cite{temam1997,wang2001,kelliher2007,bardos2007,bardos2013}, for instance in terms of the behavior of vorticity and making a connection to dissipative solutions of the Euler equations.  Other types of conditions for ensuring the inviscid limit have also been proposed \cite{constantin2015,constantin2018}. In some contexts, the unconditional inviscid limit was established \cite{sammartino1998,masmoudi1998,lopes2008,fei2018,gie2019}, but whether Kato's criterion \eqref{eqn:kato} holds unconditionally in full generality is still an open question, even in two-dimensions, where both Euler and Navier--Stokes are globally well-posed. For a more thorough discussion of the vanishing viscosity limit and boundary layers in the impermeable case we refer to  survey articles \cite{bardos2007,bardos2013,maekawa2018} and references therein.

Compared to the case of characteristic boundary conditions, the literature on the inviscid limit problem in the presence of inflow/outflow at the boundary is more limited, and the available results mostly pertain to the setting of constant injection and suction in the normal direction to the boundary. In \cite{temam2002}, Temam and Wang studied the case of a 3D channel of height $h$ with periodic boundary conditions in the horizontal variables. The boundary condition at the top and bottom of the channel is injection/suction by a constant flow $\bU = -U ^* \be _3$,  for some $U ^* > 0$, which is perpendicular to the wall. They showed that $\unu - (\uE + \uBL) \to 0$ in $L ^\infty _t L ^2 _x$ and $L ^2 _t H ^1 _x$ with optimal rate provided that the initial condition $\unu ^0 = \uE ^0$ is regular, where $\uBL$ is a boundary layer of width $\nu / U ^*$, which is thinner than Prandtl's boundary layer in the characteristic case. An extension of this result to smooth bounded domains was obtained in \cite{gie2012}. Non-constant boundary data for flows linearized around a constant shear in a channel was considered in \cite{lombardo01}.

For oblique injection and suction, Doering et al.~\cite{doering2000a,doering2000b} studied the energy dissipation with oblique suction in a 3D channel of height $h$ with $\bU = V ^* \be _1 - U ^* \be _3$ at the injection boundary and $\bU = -U ^* \be _3$ at the suction boundary (we switched the notation $U ^*$ and $V ^*$ from their paper for consistency). They show a laminar dissipation rate at the suction boundary in the small viscosity limit of the form 
\begin{align}
    \label{eqn:dsw}
    \lim _{\nu \to 0} \varepsilon _\ell = \frac{\tan \theta}2 \frac{V ^{*3}}{h} = \frac{U ^* V ^{* 2}}{2 h},
\end{align}
where $\varepsilon _\ell = {(h|\partial\Omega|)^{-1}}\langle \nu \lVert \nabla \unu \rVert _{L ^2} ^2 \rangle$ is the long-time average of the energy dissipation rate per unit mass, and $\tan \theta = {U ^*}/{V ^*}$ is the tangent of the injection angle $\theta$. 

In this work, we prove that the energy dissipation rate \eqref{eqn:dsw} essentially holds for generic flows (of the required regularity) over any domain. To state our main result, recall that $\bU \in C ^1 ([0, \infty ); C ^2(\partial \Omega))$ is given. For a given $T > 0$, we assume $\uE \in C ^1 ([0, T] \times \bar \Omega)$ is a classical solution of \eqref{eqn:euler}. We denote the strength of the suction and of the tangential component of the difference between the Euler and the Navier--Stokes solutions, respectively, by
\begin{align*}
   \Ub &= \lVert \bU \rVert _{L ^\infty ((0, T) \times \Gamma ^-)}, &
   \Vb &= \lVert \uE \tangential - \bU \tangential \rVert _{L ^\infty ((0, T) \times \Gamma ^-)},
\end{align*} 
and we let $\beta$, $\gamma$, and $K$ be positive constants such that $\Vb \le \beta \Ub$ and such that 
\begin{align}
\label{cond:smooth-lower}
    \exp \left(2 \int _0 ^T \lVert \nabla \uE (t) \rVert _{L ^\infty (\Omega)} \d t \right) &\le K , &
    \int _0 ^T \int _{\Gamma ^-} \lvert \uE \tangential - \bU \tangential \rvert ^2 U \normal \d \sigma \d t &\ge \gamma S T \Ub \Vb ^2,
\end{align}
where $S = \lvert \Gamma ^- \rvert$ is the area of the outflow boundary. The constants $\beta$, $\gamma$, and $K$ have the following meaning: $1/\beta$ bounds the averaged injection angle $\tan \theta = \bar U / \bar V$ from below,  $K$ represents the largest possible vorticity stretching rate by the Euler flow, while $\gamma$ ensures sufficient strength of the suction. Lastly, following the notation for $\Gamma_\delta$ already introduced,  $\Gammanuu$ denotes an interior neighborhood of $\Gamma ^-$ of width $\nu / \Ub$:
\begin{align*}
    \Gammanuu := \left\lbrace 
        x \in \Omega: \operatorname{dist} (x, \Gamma ^-) < \nu / \Ub
    \right\rbrace.
\end{align*}
We are now in the position to state our main result.

\begin{theorem}
    \label{thm:main}
    Fix $T > 0$. Let $\uE \in C ^1 ([0, T] \times \bar \Omega)$ be a classical solution to \eqref{eqn:euler}, and let $\unu$ be a Leray--Hopf weak solution to \eqref{eqn:Navier-Stokes}. 
    Then there exist positive constants $c _1, c _2$ depending on $\beta$, $\gamma$, and $K$, such that 
    \begin{align*}
        c _1 S T \Ub \Vb ^2 & \le \liminf _{\nu \to 0} \int _0 ^T \int _{\Gammanuu} \nu \lvert \nabla \unu \rvert ^2 \d x \d t \\
        & \le \limsup _{\nu \to 0} \int _0 ^T \int _{\Omega} \nu \lvert \nabla \unu \rvert ^2 \d x \d t \le c _2 S T \Ub \Vb ^2.
    \end{align*}
    Moreover, 
    \begin{align*}
        \limsup _{\nu \to 0} \lVert \unu - \uE \rVert _{L ^\infty (0, T; L ^2 (\Omega))} ^2 \le c _2 S T \Ub \Vb ^2.
    \end{align*}
\end{theorem}

The theorem thus asserts that the energy dissipation rate over $\Gammanuu$ is comparable to the energy dissipation rate in the entire $\Omega$; this result is consistent with the width of the boundary layer in a channel studied in \cite{temam2002,gie2012,guoyang2024}, for which an explicit construction of the boundary layer was developed under higher regularity assumptions. In fact, following the explicit expressions for the boundary layers in these particular works, one can verify that the enstrophy of the solutions obtained there is also of order $\mathcal O (\nu ^{-1})$. Furthermore, the rate $\Ub \Vb ^2 = \tan \theta\, \bar V ^3$ is consistent with the rate \eqref{eqn:dsw} obtained in \cite{doering2000a,doering2000b} in the non-tangential regime ($\tan \theta \ge \frac1\beta$). Comparing to these previous works, ours is the first of such results for non-characteristic boundary conditions established in the setting of weak Navier--Stokes solutions over general domains. Notably, our result does not rely on an explicit construction of the approximate Navier--Stokes solution in the boundary layer, which typically requires high regularity for both the Euler and Navier—Stokes solutions. Indeed, to date it is unknown whether the inviscid limit holds in the low regularity setting for Navier—Stokes solutions in which we work. The proof of \cref{thm:main} crucially relies on an estimate on the boundary velocity gradient, inspired by similar estimates obtained in \cite{vasseur2023,vasseur2024} in the impermeable case.

\begin{remark}
    If we consider the $L ^2$-average of the velocity gradient near the boundary, we have 
    \begin{align*}
        \left( 
            \fint _{[0, T] \times \Gammanuu} \lvert \nabla \unu \rvert ^2 \d x \d t
        \right) ^\frac12 \approx \frac{\Ub \Vb}\nu.
    \end{align*}
    Here the average integral means $\fint _A = \frac1{|A|}\int _A$ for any measurable set $A$. Informally, this estimate shows that the average enstrophy production on the boundary is proportional to the Reynolds number. The corresponding result is not known in the characteristic case, except in special cases, in view of Kato's condition \eqref{eqn:kato}. By contrast, we are able to obtain a lower bound for the non-characteristic case thanks to the existence of a positive energy flux, which is zero in the characteristic case.
\end{remark}

\begin{remark}
    The conclusions of \cref{thm:main} seem analogous to the notion of anomalous dissipation. Indeed, Kolmogorov's \textit{zeroth law of turbulence} \cite{kolmogoroff1941a, kolmogoroff1941b, kolmogoroff1941c} predicts that for turbulent flows, the rate of energy dissipation tends to a positive limit as $\nu \to 0$, even in the interior of $\Omega$. We refer to \cite{brue2023,cheskidov2024,brue2024}, where examples of deterministic Navier--Stokes flows in the full space or on the three-dimensional torus that exhibit anomalous dissipation were constructed in the presence of external force, while for the unforced case we mention the recent preprint \cite{burczakArXiV2025} (see also \cite{derosa2024}). For deterministic examples of anomalous dissipation in scalar turbulence, where a passive tracer is advected by a turbulent flow,  we refer to \cite{amstromg2025,colombo2023,huysmans2024,elgindi2024,drivas2022} (the corresponding zeroth law is often called Yaglom's law \cite{yaglom1949local,monin2007}). To the authors' knowledge, there are no known rigorous deterministic examples of flows exhibiting anomalous dissipation with no-slip condition at impermeable walls. In our setting instead, there is a positive rate of energy dissipation  even for laminar flows (as can be seen in the example presented in \cref{sec:eg}). This fact is mainly due to the non-characteristic boundary conditions, which allow for a positive energy flux through the boundary. Enstrophy is produced due to the suction at the boundary, which is assumed to occur at a rate independent of viscosity. However, we stress that the connection to turbulence is only by analogy,  since our results are valid for any Leray--Hopf weak solution.
\end{remark}

 The paper is organized as follows. In \cref{sec:preliminary}, we introduce the notation and the basic estimates. In \cref{sec:local}, we estimate the boundary velocity gradient in a local neighborhood of the boundary. This will be used to estimate the boundary velocity gradient globally in \cref{sec:global}. The local-to-global step requires a Calder\'on--Zygmund style decomposition, which is similar to \cite{vasseur2023,vasseur2024}, and the details are deferred to \cref{app:cz}. Finally, we conclude the paper with the proof of \cref{thm:main} in \cref{sec:main-proof}.

\section{Preliminaries}
\label{sec:preliminary}

Throughout the article, we will assume that the boundary velocity $\bU$ has a divergence-free extension into the interior of $\Omega$, such that $\bU \in C ^1 ([0, T]; C ^2(\bar \Omega))$. One may refer to \cite[Section 2.5]{temam2024} for details. Such extension is not unique, but it does not affect the estimates. Since the Navier--Stokes solutions are Leray--Hopf solutions that exists globally in time, $T>0$ is the time of existence of the Euler solution.

\subsection{Energy estimate, Leray--Hopf weak solutions, and Trace estimates}

We introduce a forcing term $\bF$ such that $\bU$ solves the Navier--Stokes equation with no pressure:
\begin{align*}
    \begin{cases}
        \partial _t \bU + \bU \cdot \nabla \bU = \nu \Delta \bU + \bF, \\
        \div \bU = 0.
    \end{cases}
\end{align*}
We note that although $\bF$ depends on $\nu$, it is uniformly bounded in $L ^\infty ((0, T) \times \Omega)$.

Let $\vnu := \unu - \bU$. Then $\vnu$ satisfies
\begin{align}
    \label{eqn:Navier-Stokes-difference}
    \begin{cases}
        \partial _t \vnu + \unu \cdot \nabla \vnu + \vnu \cdot \nabla \bU + \nabla \pnu = \nu \Delta \vnu - \bF, \\
        \div \vnu = 0, \\
        \vnu \big\vert _{\Gamma} = \boldsymbol 0.
    \end{cases}
\end{align}
Formally, if one takes the dot product of the above equation with $\vnu$, then with sufficient regularity,  the following energy equality would hold pointwise:
\begin{align*}
    \partial _t \frac{\lvert \vnu \rvert ^2}{2} + \div \left(
        \frac{\lvert \vnu \rvert ^2}{2} \unu + \pnu \vnu
    \right) + \vnu ^{\otimes 2} : \nabla \bU + \nu \lvert \nabla \vnu \rvert ^2 = \nu \Delta \frac{\lvert \vnu \rvert ^2}2 - \vnu \cdot \bF
\end{align*}
where
\begin{align*}
    \bu\otimes\bv=(\bu^i\bv^j)_{i,j},\qquad \bv ^{\otimes 2}=\bv\otimes \bv,\qquad \bA:\bB=\tr(\bA ^\top \bB),
\end{align*}
for any $d\times d$ real matrices $\bA,\bB$. That is, 
\begin{align}
    \label{eqn:weak}
    \partial _t \frac{\lvert \vnu \rvert ^2}{2} + \div \left(
        \left(\frac{\lvert \vnu \rvert ^2}{2} + \pnu\right) \vnu
    \right) + \nu \lvert \nabla \vnu \rvert ^2 = \nu \Delta \frac{\lvert \vnu \rvert ^2}2 - D (\vnu)
\end{align}
with 
\begin{align*}
    D (\vnu) = \vnu \cdot \bF + \vnu ^{\otimes 2} : \nabla \bU + \bU \cdot \nabla \frac{\lvert \vnu \rvert ^2}{2}.
\end{align*}
Even though $D(\vnu)$ depends on $\bU$, we do not explicitly indicate this dependence, as $\bU$ is fixed throughout.
Because $\vnu$ vanishes on the boundary, integrating over $\Omega$ gives  the following bound:
\begin{align*}
    &\frac{\d}{\d t} \int _\Omega \frac{\lvert \vnu \rvert ^2}2 \d \bx + \int _\Omega \nu \lvert \nabla \vnu \rvert ^2 \d \bx \\
    & \qquad = -\int _\Omega \vnu ^{\otimes 2} : \nabla \bU \d \bx - \int _\Omega \vnu \cdot \bF \d x \\
    & \qquad \le \lVert \nabla \bU \rVert _{L ^\infty (\Omega)} \lVert \vnu \rVert _{L ^2 (\Omega)} ^2 + \lVert \bF \rVert _{L ^2 (\Omega)} \lVert \vnu \rVert _{L ^2 (\Omega)} \\
    & \qquad \le \lVert \nabla \bU \rVert _{L ^\infty (\Omega)} \lVert \vnu \rVert _{L ^2 (\Omega)} ^2 + T ^{-1} \lVert \vnu \rVert _{L ^2 (\Omega)} ^2 + T \lVert \bF \rVert _{L ^\infty (0, T; L ^2 (\Omega))} ^2 \\
    & \qquad \le \left(
        \lVert \nabla \bU \rVert _{L ^\infty (\Omega)} + T ^{-1} 
    \right) \left( 
        \lVert \vnu \rVert _{L ^2 (\Omega)} ^2 + T ^2 \lVert \bF \rVert _{L ^\infty (0, T; L ^2 (\Omega))} ^2
    \right).
\end{align*}
We next define $\tE = T ^2 \lVert \bF \rVert _{L ^\infty (0, T; L ^2 (\Omega))} ^2$, and note that it has the units of energy. Then using Gr\"onwall's inequality, we conclude for every $t \in (0, T)$ that 
\begin{align*}
    & \frac12 \lVert \vnu (t) \rVert _{L ^2 (\Omega)} ^2 + \nu \lVert \nabla \vnu \rVert _{L ^2 ((0, t) \times \Omega)} ^2 + \tE 
    \\
    & \qquad
    \le C \left(
        \lVert \vnu (0) \rVert _{L ^2 (\Omega)} ^2 + \tE
    \right) \exp \left(
        2 \lVert \nabla \bU \rVert _{L ^1 (0, t; L ^\infty (\Omega))} 
    \right).
\end{align*}
Therefore, by taking the supremum in $t \in [0, T]$ it follows that 
\begin{align*}
    & \lVert \vnu \rVert _{L ^\infty (0, T; L ^2 (\Omega))} ^2 + \nu \lVert \nabla \vnu \rVert _{L ^2 ((0, T) \times \Omega)} ^2 + \tE \le C (\lVert \vnu (0) \rVert _{L ^2 (\Omega)}, \bU, T) =: \KE,
\end{align*}
where we used that  $\vnu (0)$ is uniformly bounded in $L ^2$ in $\nu$, because $\unu ^0 \to \uE ^0$ in $L ^2 (\Omega)$.
We may choose $\KE$ sufficiently large (by choosing the data $\uE$ and $\bU$ appropriately) such that 
\begin{align*}
    \lVert \bU \rVert _{L ^\infty (0, T; L ^2 (\Omega))} ^2 + \nu \lVert \nabla \bU \rVert _{L ^2 ((0, T) \times \Omega)} ^2 \le \KE,
\end{align*}
 from which we  deduce that 
\begin{align}\label{est:energy-inequality}
    & \lVert \unu \rVert _{L ^\infty (0, T; L ^2 (\Omega))} ^2 + \nu \lVert \nabla \unu \rVert _{L ^2 ((0, T) \times \Omega)} ^2 \le 4 \KE.
\end{align}
Therefore $\{\unu\} _\nu$ is a uniformly bounded sequence in $L ^\infty (0, T; L ^2 (\Omega))$.

Thanks to \eqref{eqn:weak}, in standard fashion we can construct global weak solutions to \eqref{eqn:Navier-Stokes-difference} satisfying the energy inequality; that is, for every $t \in [0, T]$, it holds that 
\begin{align}
    \label{eqn:energy-inequality}
    \int _\Omega \frac{\lvert \vnu (t) \rvert ^2}2 \d \bx + \int _0 ^t \int _\Omega \nu \lvert \nabla \vnu \rvert ^2 \d \bx \d s & \le \int _\Omega \frac{\lvert \vnu (0) \rvert ^2}2 \d \bx  \\
    \notag
    & \qquad - \int _0 ^t \int _\Omega \vnu ^{\otimes 2} : \nabla \bU + \vnu \cdot \bF \d x \d s,
\end{align}
which yields a corresponding energy inequality for $\unu$.
Therefore, for the rest of the paper we assume $\unu$ is a \textit{Leray--Hopf weak solution}, that is, 
\begin{align*}
    \unu \in C _{\mathrm w} (0, T; L ^2 (\Omega)) \cap L ^2 (0, T; \dot H ^1 (\Omega))
\end{align*}
is a weak solution to \eqref{eqn:Navier-Stokes} in the sense of distribution such that $\vnu = \unu - \bU$ satisfies the energy inequality \eqref{eqn:energy-inequality}.

We remark that Leray--Hopf solutions have a valid notion of \textit{trace} in $L ^\frac43 _{\mathrm{loc}} ((0, T] \times \Gamma)$; local integrability appears in order to accommodate the low regularity of the initial data for Navier--Stokes. We work with $\vnu$ first. In dimension $d = 3$, by interpolation and the Sobolev Embedding, 
\begin{align*}
    \lVert \vnu \rVert _{L ^4 (0, T; L ^3 (\Omega))} ^2 &\le \lVert \vnu \rVert _{L ^\infty (0, T; L ^2 (\Omega))} \lVert \vnu \rVert _{L ^2 (0, T; L ^6 (\Omega))} \\
    &\le C \lVert \vnu \rVert _{L ^\infty (0, T; L ^2 (\Omega))} \lVert \nabla \vnu \rVert _{L ^2 (0, T; L ^2 (\Omega))} \le 
    C \KE\nu ^{-\frac12}.
\end{align*}
Similarly,
\begin{align*}
    \lVert \unu \rVert _{L ^4 (0, T; L ^3 (\Omega))} ^2 
    & \le C \KE\nu ^{-\frac12}.
\end{align*}

Let $\rho: [0, T] \to [0, 1]$ be a smooth cutoff function such that $\rho (0) = 0$, and $\rho (t) = 1$ for $t \in [T / 2, T]$, with $|\rho'| \le 3 / T$. Then 
\begin{align*}
    \begin{cases}
        \partial _t (\rho \vnu) + \nabla (\rho \pnu) = \nu \Delta (\rho \vnu) - \rho (\vnu \cdot \nabla \bU + \unu \cdot \nabla \vnu + \bF) + \rho' \vnu \\
        \div (\rho \vnu) = 0 \\
        \rho \vnu \big\vert _{\Gamma} = \boldsymbol 0 \\
        \rho \vnu \big\vert _{t = 0} = \boldsymbol 0.
    \end{cases}
\end{align*}
Using estimates for the Cauchy problem of the (time-dependent) Stokes system, we know 
\begin{align*}
    & \nu \lVert \nabla ^2 \vnu \rVert _{L ^\frac43 (T / 2, T; L ^\frac65 (\Omega))} + \lVert \nabla \pnu \rVert _{L ^\frac43 (T / 2, T; L ^\frac65 (\Omega))} \\
    & \qquad \lesssim 
    \lVert \vnu \cdot \nabla \bU + \unu \cdot \nabla \vnu + \bF \rVert _{L ^\frac43 (0, T; L ^\frac65 (\Omega))} + \frac1T \lVert \vnu \rVert _{L ^\frac43 (0, T; L ^\frac65 (\Omega))} 
    \\
    & \qquad \lesssim \lVert \vnu \rVert _{L ^4 (0, T; L ^3 (\Omega))} \lVert \nabla \bU \rVert _{L ^2 (0, T; L ^2 (\Omega))} + \lVert \unu \rVert _{L ^4 (0, T; L ^3 (\Omega))} \lVert \nabla \vnu \rVert _{L ^2 (0, T; L ^2 (\Omega))} \\
    & \qquad \qquad + T ^\frac34 |\Omega| ^\frac13 \lVert \bF \rVert _{L ^\infty (0, T; L ^2 (\Omega))} + T ^{-\frac14} |\Omega|^\frac13 \lVert \vnu \rVert _{L ^\infty (0, T; L ^2 (\Omega))} \\
    & \qquad \lesssim \KE \nu ^{-\frac34} + \KE^{\frac12}|\Omega|^{\frac13}T ^{-\frac14}
\end{align*}

Here we used $\tE \le \KE$, and $f \lesssim g$ means $f \le c g$ for some universal non-dimensional constant $c$. By the Trace Theorem and the Sobolev Embedding, we control $\vnu$ in $L ^\frac43$ on the boundary:
\begin{align}\label{est:trace}
    \nu \lVert \nabla \vnu \rVert _{L ^\frac43 ((T / 2, T) \times \Gamma)} \lesssim \KE \nu ^{-\frac34} + \KE^{\frac12}|\Omega|^{\frac13}T ^{-\frac14}.
\end{align}

This is true for any $T$, so we can combine the estimates in $[\frac T2, T]$, $[\frac T4, \frac T2]$, \dots, which yields that for any $t _0 > 0$, it holds that
\begin{align}\label{est:trace-2}
    \nu \lVert \nabla \vnu \rVert _{L ^\frac43 ((t _0, T) \times \Gamma)} \lesssim \KE\left[\frac1\nu \log \left(\frac T{t _0} \right)\right] ^{\frac34} + \KE^{\frac12}|\Omega|^{\frac13} t _0 ^{-\frac14}.
\end{align}
The trace of $\nabla \unu$ is also well-defined now because it differs from $\nabla \vnu$ by $\nabla \bU$ which is continuous. Subsequently, it satisfies an estimate of the form \eqref{est:trace} as well.

In dimension $d = 2$, \eqref{est:trace} can be improved by estimates in $L ^\frac32 ((t _0, T) \times \Gamma)$, and we get 
\begin{align*}
    \nu \lVert \nabla \vnu \rVert _{L ^\frac32 ((T/2, T) \times \Gamma)} \lesssim \KE \nu ^{-\frac23} + \KE ^{\frac12} |\Omega| ^{\frac13} T ^{-\frac13},
\end{align*}
which implies \eqref{est:trace} with $\Omega$-dependent constant. One can also add a third dimension $x _3 \in \mathbb T$ and work in $\tilde \Omega = \Omega \times \mathbb T$ to obtain the same result. Indeed, \eqref{est:trace} is still true in 2D as long as we allow the constant to depend on the length scale.

\subsection{Estimating the size of the boundary layer}
\label{sec:boundary-layer}

In this subsection, we estimate the $L ^2$ difference between the Navier--Stokes solution and the Euler solution. In particular, we derive the inequality \eqref{eqn:layer-separation-rate-2} below, which ultimately serves the departure point in obtaining the boundary layer and energy dissipation rate estimates claimed in \cref{thm:main}.

Formally, with sufficient regularity, the difference between $\unu$ and $\uE$ is given by the following formula:
\begin{align}
    \notag 
    & \hspace{-2em} \frac{\d}{\d t} \int _\Omega \frac{\lvert \unu - \uE \rvert ^2}2 \d x 
    + \int _\Omega \nu \lvert \nabla \unu \rvert ^2 \d x 
    + \int _{\partial \Omega} \frac{\lvert \unu - \uE \rvert ^2}2 (\unu \cdot \bn) \d \sigma \\
    \label{eqn:layer-separation-rate-1}
    & = -\int _\Omega [(\unu - \uE) \otimes (\unu - \uE)] : \nabla \uE \d x - \int _{\partial \Omega} (\pnu - \pE) (\unu - \uE) \cdot \bn  \d \sigma \\
    \notag 
    & \qquad + \nu \int _\Omega \nabla \unu : \nabla \uE \d x - \nu \int _{\partial \Omega} \partial _n \unu \cdot (\unu - \uE) \d \sigma.
\end{align}
Since $\unu = \bU$ and $(\unu - \uE) \cdot \bn = 0$ on $\partial \Omega$, $\uE = \bU$ on $\Gamma ^+$, it follows that
\begin{align*}
    &\int _{\partial \Omega} \frac{\lvert \unu - \uE \rvert ^2}2 (\unu \cdot \bn) \d \sigma = \frac12 \int _{\Gamma ^-} \lvert \bU \tangential - \uE \tangential \rvert ^2 U \normal \d \sigma, \\
    &\int _{\partial \Omega} (\pnu - \pE) (\unu - \uE) \cdot \bn \d \sigma = 0, \\
    &\nu \int _{\partial \Omega} \partial _n \unu \cdot (\unu - \uE) \d \sigma = \nu \int _{\Gamma ^-} \partial _n \unu \cdot (\bU \tangential - \uE \tangential) \d \sigma.
\end{align*}
So \eqref{eqn:layer-separation-rate-1} becomes
\begin{align*}
    & \frac{\d}{\d t} \int _\Omega \frac{\lvert \unu - \uE \rvert ^2}2 \d x 
    + \int _\Omega \nu \lvert \nabla \unu \rvert ^2 \d x 
    + \frac12 \int _{\Gamma ^-} \lvert \bU \tangential - \uE \tangential \rvert ^2 U \normal \d \sigma \\
    & \qquad = -\int _\Omega [(\unu - \uE) \otimes (\unu - \uE)] : \nabla \uE \d x \\
    & \qquad \qquad + \nu \int _\Omega \nabla \unu : \nabla \uE \d x + \nu \int _{\Gamma ^-} \partial _n \unu \cdot (\uE \tangential - \bU \tangential) \d \sigma \\
    & \qquad \le \lVert \nabla \uE \rVert _{L ^\infty (\Omega)} (t) \int _\Omega \lvert \unu - \uE \rvert ^2 \d x \\
    & \qquad \qquad + \frac\nu2 \int _\Omega \lvert \nabla \unu \rvert ^2 \d x + \frac\nu2 \int _\Omega \lvert \nabla \uE \rvert ^2 \d x + \nu \int _{\Gamma ^-} \partial _n \unu \cdot (\uE \tangential - \bU \tangential) \d \sigma.
\end{align*}
By integrating between $0$ and $t$, we have
\begin{align}
    \notag
    & \hspace{-2em} \lVert \unu - \uE \rVert _{L ^2 (\Omega)} ^2 (t) 
    + \nu \int _0 ^t \lVert \nabla \unu \rVert _{L ^2 (\Omega)} ^2 (s) \d s 
    + \int _0 ^t \int _{\Gamma ^-} \lvert \uE - \bU \rvert ^2 U \normal \d \sigma \d s \\
    \label{eqn:layer-separation-rate-2}
    & \le \lVert \unu - \uE \rVert _{L ^2 (\Omega)} ^2 (0) + 2 \int _0 ^t \lVert \nabla \uE (s) \rVert _{L ^\infty (\Omega)} \lVert \unu - \uE \rVert _{L ^2 (\Omega)} ^2 (s) \d s \\
    \notag
    & \qquad + \nu \int _0 ^t \lVert \nabla \uE \rVert _{L ^2 (\Omega)} ^2 (s) \d s + 2 \nu \int _0 ^t \int _{\Gamma ^-} \partial _n \unu \cdot (\uE \tangential - \bU \tangential) \d \sigma \d s.
\end{align}
From the hypotheses on $\uE$ it follows that $\nabla \uE \in L ^1 (0, T; L ^\infty (\Omega)) \cap L ^2 (0, T; L ^2 (\Omega))$. 

This formal argument can be justified even when $\unu$ is only a Leray--Hopf weak solution \eqref{eqn:energy-inequality}. To have the last integral well-defined, we note that $\partial _n \unu$ has a trace away from time $0$, and the integral near $0$ is negligible (see \cref{sub:first-epoch}).

We recall that $U \normal > 0$ on $\Gamma ^-$ so that each term on the left-hand side is non-negative. The first and the third term on the right will vanish as $\nu \to 0$. The second term on the right will be taken care of by Gr\"onwall's inequality. Thus the main difficulty is to estimate the fourth and final term on the right-hand side, which is the pairing between the trace of the Navier--Stokes velocity gradient and the difference in the tangential components between Euler and Navier--Stokes. This term is treated in \cref{sec:local} and \ref{sec:global}.

\subsection{A canonical example}
\label{sec:eg}

Before proceeding to the general analysis, we consider an enlightening reference case provided by the setting of the half-space. In this setting, one is able to discern the fundamental differences between characteristic and non-characteristic boundary conditions in an explicit manner, and subsequently set some basic expectations that motivates our main result. Indeed, let \begin{align*}
    \Omega = \mathbb T ^{d - 1} \times \mathbb R _+ = \{x = (x', x _d): x' \in \mathbb T ^{d - 1}, x _d > 0\}.     
\end{align*}
We then exert a constant suction on the boundary 
\begin{align*}
    \bU = V ^* \be _1 - U ^* \be _d \text{ on } \Gamma ^- = \mathbb T ^{d - 1} \times \{0\},
\end{align*}
where $V ^* \ge 0$ is the strength of the tangential velocity, interpreted as a boundary shear, and $U ^* \ge 0$ is the strength of the suction in the normal direction. We consider the idealized situation where injection happens at the infinite horizon:
\begin{align*}
    \lim _{x _d \to \infty} \bU = - U ^* \be _d.
\end{align*}
We remark that $U ^*$ and $V ^*$ are defined differently from \cite{doering2000a,doering2000b}. In this setting, the Euler solution is  given by the constant flow
\begin{align*}
    \uE = - U ^* \be _d.
\end{align*}
For simplicity, suppose $\unu ^0 = \uE = - U ^* \be _d$. We will illustrate below that the solution corresponding to $U ^* = 0$ behaves differently from the solution corresponding to  $U ^* > 0$ as $\nu \to 0$.

If $U ^* = 0$, then $\uE\equiv 0$, and the corresponding Navier--Stokes solution is given by the Prandtl solution (see for instance \cite[section 2.2]{chorin1990})
\begin{align*}
    \unu = \operatorname{erfc} \left(\frac{x _d}{\sqrt{4 \nu t}}\right) V ^* \be _1,
\end{align*}
where $\operatorname{erfc}$ is the complementary error function defined by 
\begin{align*}
    \operatorname{erfc} (z) = \frac2{\sqrt \pi} \int _z ^\infty e ^{-\xi ^2} \d \xi.
\end{align*}
From the expression of $\unu$, we see the boundary layer has width $\cO (\sqrt{\nu t})$ and the energy is given by 
\begin{align*}
    \frac12 \int _\Omega \lvert \unu - \uE \rvert ^2 \d x = \frac12 S V ^{*2} \int _0 ^\infty \operatorname{erfc} ^2 \left(\frac{x _d}{\sqrt{4 \nu t}}\right) \d x _d = \frac{2 - \sqrt 2}{\sqrt \pi} \sqrt{\nu t} S V ^{*2},
\end{align*}
which tends to zero at a rate of $\cO (\sqrt{\nu t})$ as $\nu \to 0$. As for the energy dissipation, we observe that the enstrophy at time $t$ is given by 
\begin{align*}
    \int _\Omega \lvert \nabla \unu \rvert ^2 \d x = S V ^{*2} \int _0 ^\infty \left(\frac1{\sqrt{4 \nu t}} \operatorname{erfc}' \left(\frac{x _d}{\sqrt{4 \nu t}}\right)\right) ^2 \d x _d = \sqrt{\frac2\pi}\frac{S V ^{*2}}{\sqrt{4 \nu t}},
\end{align*}
so the energy dissipation between $0$ and $t$ is 
\begin{align*}
    \int _0 ^t \int _\Omega \nu \lvert \nabla \unu \rvert ^2 \d x \d s = \sqrt{\frac2\pi} \sqrt{\nu t} S V ^{*2},
\end{align*}
which also tends to zero at the rate of $\cO (\sqrt{\nu t})$ if $\nu \to 0$. In fact, the long-time average energy dissipation rate is zero uniformly in viscosity: $\langle \nu\lVert\nabla \unu\rVert^2_{L^2(\Omega)}\rangle=\lim_{t \to \infty}\frac{1}t\int_0^t\nu\lVert\nabla\unu(s)\rVert^2_{L^2(\Omega)}\d s=0$.

Now let us study the case when $U ^* > 0$. In this scenario, the corresponding Navier--Stokes equation has a \textit{stationary} solution given by 
\begin{align*}
    \unu ^\infty = \exp \left(-\frac{U ^* x _d}{\nu}\right) V ^* \be _1 - U ^* \be _d.
\end{align*}
Then, making the ansatz $\unu (t, x) = \unu ^\infty (x) + f (t, x _d) V ^* \be _1$ and $\pnu = 0$, $f (t, z)$ satisfies the following initial boundary value problem:
\begin{align*}
    \begin{cases}
        \partial _t f - U ^* \partial _z f = \nu \partial _{zz} f & t > 0, z > 0 \\
        f (0, z) = - \exp \left( -\frac{U ^* z}\nu \right) & z > 0 \\
        f (t, 0) = 0 & t > 0 \\
        \lim _{z \to \infty} f (t, z) = 0 & t > 0 \,.
    \end{cases}
\end{align*}
We can interpret the term $f (t, x _d) V ^* \be _1$ as a suitable corrector.
The system above is explicitly solvable. We let $g (t, z) = \exp (\frac{U ^* z}{2 \nu} + \frac{U ^{* 2} t}{4 \nu}) f (t, z)$. Then $g$ satisfies the heat equation in half line with Dirichlet boundary condition:
\begin{align*}
    \begin{cases}
        \partial _t g = \nu \partial _{zz} g & t > 0, z > 0 \\
        g (0, z) = - \exp \left( -\frac{U ^* z}{2 \nu} \right) & z > 0 \\
        g (t, 0) = 0 & t > 0 \\
        \lim _{z \to \infty} \exp (-\frac{U ^* z}{2 \nu}) g (t, z) = 0 & t > 0 \,.
    \end{cases}
\end{align*}
Consequently, the solution is given in terms of the odd extension $\tilde{g}$ of $g$, which solves the heat equation in the full plane with initial data 
\begin{align*}
    \tilde{g}_0 (z) := - \operatorname{sgn} (z) 
    \exp \left( -\frac{U ^* |z|}{2 \nu} \right),
\end{align*}
and hence it is explicitly given by $\tilde g (t, z) = g _0 * H _{\nu t} (z)$, where
\begin{align*}
    H _{\nu t} (z) = \frac1{\sqrt{4 \pi \nu t}} \exp \left(-\frac{z ^2}{4 \nu t}\right)
\end{align*}
is the heat kernel. Therefore, 
\begin{align*}
    \unu (t, x) &= \unu ^\infty (t, x) + \exp \left(-\frac{U ^* x _d}{2 \nu} - \frac{U ^{*2} t}{4 \nu}\right) (g _0 * H _{\nu t}) (x _d) V ^* \be _1 \\
    &= \frac12 \left(
        \exp \left(-\frac{U ^* x _d}{\nu}\right) \operatorname{erfc} \left(\frac{x _d - U ^* t}{\sqrt{4 \nu t}}\right)+ \operatorname{erfc} \left(\frac{x _d + U ^* t}{\sqrt{4 \nu t}}\right) 
    \right) V ^* \be _1 - \bU ^* \be _d.
\end{align*}
The heat equation dissipates kinetic energy, so $\lVert g _0 * H _{\nu t} \rVert _{L ^2 (\mathbb R _+)} \le \lVert g _0 \rVert _{L ^2 (\mathbb R _+)}$. Therefore, as long as $U ^* > 0$, we see that $\unu (t) \to \unu ^\infty$ converges strongly in $L ^2$ at an exponential rate as $t \to +\infty$. 

\begin{figure}[htbp]
    \centering
    \begin{tikzpicture}
        \def\width{6}
        \def\height{3}

        \foreach \x in {0.1,0.3,...,\width}{
            \draw (\x,0) -- (\x-0.2,-0.2);
        }

        \draw[dashed] (\width/2, \height) -- (\width/2, 0);
        \draw[color=red,domain=0.1:2] plot (\x + \width/2, {5*exp(-\x)-5*exp(-2)-1});

        \draw (0,0) -- (\width,0) node[right] {$\Gamma ^-$};

        \draw[-latex] (\width/2, 0) -- (\width/2 + 2, -1) node[anchor = north west] {$\bU$};

        \draw[-latex] (\width/2, {5*exp(-1)-5*exp(-2)}) -- (\width/2 + 1, {5*exp(-1)-5*exp(-2)-1});

        \draw[-latex] (\width/2, {5*exp(-0.5)-5*exp(-2)}) -- (\width/2 + 0.5, {5*exp(-0.5)-5*exp(-2)-1});

        \draw[thick,dotted] (\width/2, 0) -- (\width/2, -1) node[midway,anchor = east] {$U ^*$} -- (\width/2 + 2, -1) node[midway,anchor = north] {$V ^*$};
        \draw (\width/2 + 2 -.5, -1) arc (180:153.53:.5) node[midway,anchor=east,yshift=2pt] {$\theta$};

        \usetikzlibrary{decorations.pathreplacing}
        \draw [decorate,decoration={brace,amplitude=5pt}] (\width/2, 0) -- (\width/2,2) node [midway,anchor = east,xshift=-5pt] {$\nu/{U ^*}$};

    \end{tikzpicture}
    \caption{Stationary solution to the Navier--Stokes equation with inflow/outflow boundary condition.}
    \label{fig:exponential}
\end{figure}

Let us now focus on the stationary solution $\unu ^\infty$. We plot the stationary solution in \cref{fig:exponential}. The boundary layer has width $\cO (\nu / U ^*)$, and
\begin{align*}
    \frac12 \int _\Omega \lvert \unu ^\infty - \uE \rvert ^2 \d x = \frac12 \int _\Omega \left({V ^*} \exp \left(-\frac{U ^* x _d}{\nu}\right)\right) ^2 \d x _d = \frac{\nu S}4 \frac{V ^{*2}}{U ^*},
\end{align*}
which vanishes of order $\cO (\nu)$ as $\nu \to 0$. However, the rate of energy dissipation for $\unu ^\infty$ does not vanish:
\begin{align*}
    \nu \int _\Omega \lvert \nabla \unu ^\infty \rvert ^2 \d x = \nu S \int _0 ^\infty \left(\frac{V ^* U ^*}{\nu} \exp \left(-\frac{U ^* x _d}{\nu}\right)\right) ^2 \d x _d = \frac12 S U ^* V ^{*2}.
\end{align*}
In this paper we will consider a general domain and a general flow. Although we do not address the issue of whether $\unu \to \uE$ as $\nu \to 0$ uniformly in $L ^2$, we do show that the average rate of energy dissipation is bounded away from zero in the limit of vanishing viscosity.

\section{Local behavior of the velocity gradient near the boundary}
\label{sec:local}

As pointed out in \cref{sec:boundary-layer}, the main difficulty is to estimate the velocity gradient near the boundary. This will be done in two steps. In this section, we will localize the system near a point on the boundary and estimate the velocity gradient in a small neighborhood. Since we will only need to pair the velocity gradient with the regular vector field $\uE - \bU$, we only need to control the velocity gradient in an averaged sense. In the next section, we will estimate the velocity gradient in a global sense. This ``local-to-global'' blow-up argument is inspired by \cite{vasseur2010,choi2014}, and has also been developed for higher derivative estimates \cite{vasseur2021} and interior trace estimates \cite{yang2023}.

To motivate the local estimate, let us consider the simple scenario of a flat boundary. The case of curved boundary is more technical, but similar ideas apply.

We localize our system near a point $\bx ^* \in \Gamma ^-$, the outflow boundary, and we fix a time $t ^* > 0$. Given any $\varepsilon > 0$, we construct a box $B ^+ _\varepsilon (\bx ^*) = \bar B _\varepsilon (\bx ^*) \times [0, \varepsilon)$, the base of which $\bar B _\varepsilon (\bx ^*) \subset \Gamma ^-$ is a $(d - 1)$-dimensional box, and we denote $Q _\varepsilon ^+ (t ^*, \bx ^*) = (t ^* - \varepsilon ^2 / \nu, t ^*) \times B _\varepsilon ^+$, $\bar Q _\varepsilon = (t ^* - \varepsilon ^2 / \nu, t ^*) \times \bar B _\varepsilon$. This box is chosen such that upon scaling, the rescaled solution solves Navier--Stokes equation in a unit cube with viscosity $1$.

\begin{figure}[htbp]
    \centering    
    \begin{tikzpicture}
        \def\width{6}
        \def\height{3}

        \draw (1, 0) -- (1, 2.472) -- (5, 2.472) node [anchor=west] {$B _{2\varepsilon} ^+$} -- (5, 0);
        \draw (2, 0) -- (2, 1.236) -- (4, 1.236) node [anchor=north east] {$B _{\varepsilon} ^+$} -- (4, 0);

        \foreach \x in {0.1,0.3,...,\width}{
            \draw (\x,0) -- (\x-0.2,-0.2);
        }

        \draw (0,0) -- (\width,0) node[right] {$\Gamma ^-$};
        \draw[-latex] (\width/2 - 0.3, 0) -- (\width/2 + 0.3, -1) node[anchor = south west] {$\bU$};
        \fill (\width/2, 0) node[anchor = south] {$\bx ^*$} circle (.05);

    \end{tikzpicture}
    \caption{Localization near the boundary at space-scale $\varepsilon$ and time-scale $\varepsilon ^2 / \nu$.}
    \label{fig:local-boundary}
\end{figure}

For notational convenience, we let $\lVert \cdot \rVert _\infty$ represent the $L ^\infty ((0, T) \times \Omega)$ norm. Then we prove the following local estimate. 

\begin{lemma}
    \label{lem:local}
    Let $t ^* > 0$, $\bx ^* \in \Gamma ^-$. 
    Recall that $\bU \in C ^1([0, T]; C ^2(\bar \Omega))$. Given $\eta \in (0, 1)$, let $\nu > 0$ be sufficiently small such that the following conditions are satisfied:
    \begin{align*}
        \nu \lVert \nabla \bU \rVert _\infty \le \eta \Ub ^2, \qquad \nu \lVert \partial _t \bU \rVert _\infty + \nu ^2 \lVert \Delta \bU \rVert _\infty \le \eta \Ub ^3.
    \end{align*}
    Let $\varepsilon \in (0, \nu / (2 \Ub))$. If $\unu$ is a weak solution to the Navier--Stokes system in $Q _{2 \varepsilon} ^+ (t ^*, \bx ^*)$ with $\unu = \bU$ on $\Gamma ^-$ and  
    \begin{align*}
        \left( 
            \fint _{Q _{2 \varepsilon} ^+} \lvert \nabla \unu \rvert ^2 \d x \d t
        \right) ^\frac12 \le \eta \nu \varepsilon ^{-2},
    \end{align*}
    then there exists a universal constant $c$ such that 
    \begin{align*}
        \left\lvert 
            \fint _{\bar Q _{\varepsilon}} \partial _n \unu \d \sigma \d t
        \right\rvert \le c\eta \nu \varepsilon ^{-2}.
    \end{align*}
\end{lemma}

\begin{proof}
    We make a change of variable. For $(s, \by) \in (-4, 0) \times B _2 ^+$, define the non-dimensional variables
    \begin{align*}
        \bv (s, \by) &= \frac{(\unu - \bU) (t ^* + \varepsilon ^2 \nu ^{-1} s, \bx ^* + \varepsilon \by)}{\nu / \varepsilon}, \\
        q (s, \by) &= \frac{\pnu (t ^* + \varepsilon ^2 \nu ^{-1} s, \bx ^* + \varepsilon \by)}{\nu ^2 / \varepsilon ^2},
    \end{align*}
    Then $\bv$ satisfies the following system:
    \begin{align*}
        \begin{cases}
            \partial _s \bv + \nabla _{\by} q = \Delta _{\by} \bv - \bF & \text{ in } (-4, 0) \times B _2 ^+ \\
            \bv = \boldsymbol 0 & \text{ on } (-4, 0) \times \bar B _2.
        \end{cases}
    \end{align*}
    where 
    \begin{align*}
        \bF = \bv \cdot \nabla _{\by} \bv + \frac{\varepsilon}\nu \bU \cdot \nabla _{\by} \bv + \frac{\varepsilon ^2}\nu {\bv} \cdot \nabla _{\bx} \bU + \frac{\varepsilon ^3}{\nu ^2} (\partial _t \bU + \bU \cdot \nabla _{\bx} \bU - \nu \Delta _{\bx} \bU).
    \end{align*}
    As a reference, we list the scale of each quantity in the following table.
    \begin{table}[htbp]
        \centering
        \caption{The scale of each quantity}
        \label{tab:scale}
        \renewcommand{\arraystretch}{1.3}
        \begin{tabular}{cc|cccccc}
            \hline \hline 
            $x$ & $t$ & $\unu$ & $\nabla \unu$ & $\Delta \unu$ & $\partial _t \unu$ & $\pnu$ & $\nabla \pnu$ \\
            \hline 
            $\varepsilon$ & $\varepsilon ^2 / \nu$ & $\nu / \varepsilon$ & $\nu / \varepsilon ^2$ & $\nu / \varepsilon ^3$ & $\nu ^2 / \varepsilon ^3$ & $\nu ^2 / \varepsilon ^2$ & $\nu ^2 / \varepsilon ^3$ \\
            \hline \hline
        \end{tabular}
    \end{table}

    By the bound on $\unu$ we find that
    \begin{align*}
        \fint _{Q _{2} ^+} \lvert \nabla _{\by} \bv \rvert ^2 \d y \d s = \nu ^{-2} \varepsilon ^4 \fint _{Q _{2 \varepsilon} ^+} \lvert \nabla (\unu - \bU) \rvert ^2 \d x \d t \le 2 \eta ^2 + \frac{2 \varepsilon ^4}{\nu ^2} \lVert \nabla \bU \rVert _\infty ^2 \le 4 \eta ^2.
    \end{align*}
    Recall that $\bv = \bf 0$ on $(-4, 0) \times \bar B _2$. By the Sobolev Embedding, we also have 
    \begin{align*}
        \fint _{-4} ^0 \lVert \bv (s) \rVert _{L ^6 (B _2 ^+)} ^2 \d s \le C \fint _{-4} ^0 \lVert \nabla _{\by} \bv (s) \rVert _{L ^2 (B _2 ^+)} ^2 \d s \le c \eta ^2.
    \end{align*}
    Given the hypotheses of the Lemma, from these estimates, we conclude 
    \begin{align*}
        \fint _{-4} ^0 \lVert \bF \rVert _{L ^\frac32 (B _1 ^+)} \d s \le c \eta ^2 + c \eta \le c \eta.
    \end{align*}
    We complete the proof by using \cite[proposition 2]{vasseur2024}:
    \begin{align*}
        \left\lvert \fint _{\bar Q _{\varepsilon}} \partial _n \bv \d \sigma \d t \right\rvert \le C \left( 
            \lVert \nabla \bv \rVert _{L ^2 (Q _1 ^+)} + \lVert \bF \rVert _{L ^1 _t L ^\frac65 _x (Q _1 ^+)}
        \right) \le c \eta,
    \end{align*}
    which scales to the desired estimate, since $\lVert \partial _n \bU \rVert _\infty \le \eta \nu ^{-1} \Ub ^2 \le \eta \nu \varepsilon ^{-2}$.
\end{proof}

\section{Boundary velocity gradient estimates}
\label{sec:global}

In this section, we estimate the the normal derivative of the velocity at the outflow boundary.

\begin{proposition}
    \label{prop:boundary-work}
    Let $\unu$ be a Leray--Hopf weak solution to \eqref{eqn:Navier-Stokes}, and let $\bV = \uE - \bU$, where $\uE \in C ^1 ([0, T] \times \bar \Omega)$ is the classical solution to \eqref{eqn:euler}. Suppose $\nu < T \Ub ^2 / 4$, then for any $t \in [0, T]$, it holds that
    \begin{align*}
        \left\lvert 
            \nu \int _0 ^t \int _{\Gamma ^-} \partial _n \unu \cdot \bV \d \sigma \d s
        \right\rvert &\lesssim \Vb \left( 
            \nu \int _0 ^t \int _{\Gammanuu} \lvert \nabla \unu \rvert ^2 \d x \d s
        \right) ^\frac23 \left(
            T S
        \right) ^\frac13 \\
        & \quad + \Ub ^\frac12 \Vb \left( 
            \nu \int _0 ^t \int _{\Gammanuu} \lvert \nabla \unu \rvert ^2 \d x \d s
        \right) ^\frac12 \left(
            T S
        \right) ^\frac12 + R _\nu,
    \end{align*}
    with $R _\nu \to 0$ uniform in $t$ as $\nu \to 0$.
\end{proposition}

We observe that by construction $\bV$ is divergence free and $V \normal = 0$.
To prove this proposition, we set $T _2 := \nu \Ub ^{-2} / 4$ and $T _1 := T _2 ^2 / T = \cO (\nu ^2)$. From the assumption $\nu < T \Ub ^{-2}$ we have $T _1 < T _2 < T$. Then the time interval $[0, T]$ is separated into three parts:
\begin{align*}
    [0, T] = [0, T _1] \cup [T _1, T _2] \cup [T _2, T],
\end{align*}
The proof of \cref{prop:boundary-work} will be accomplished in three cases: $t \in [0, \Tone]$, $t \in [\Tone, \Ttwo]$, and $t \in [\Ttwo, T]$. We separate each case in three subsections, using the following key lemma.

\begin{lemma} 
    \label{lem:boundary-work}
    Let $\tau < \nu \Ub ^{-2} / 4$ be positive. 
    If $\unu$ is a Leray--Hopf weak solution to \eqref{eqn:Navier-Stokes} in $(\tau / 4, 5 \tau / 4) \times \Omega$, then
    \begin{align*}
        \left\lvert 
            \nu \int _{\tau} ^{5 \tau / 4} \int _{\Gamma ^-} \partial _n \unu \cdot \bV \d \sigma \d s
        \right\rvert &\lesssim \kappa \nu ^2 \left(
            \int _{\tau} ^{5 \tau / 4} \int _{\Gamma ^-} \lvert \partial _n \unu \rvert ^\frac43 \d \sigma \d s
        \right) ^\frac34 \tau ^\frac14 S ^{\frac14} \\
        & \qquad + \left( 
            \nu \int _{\tau / 4} ^{5 \tau / 4} \int _{\Gammanuu} \lvert \nabla \unu \rvert ^2 \d x \d s
        \right) ^\frac23 \Vb \tau  ^\frac13 S ^\frac13 \\
        & \qquad + \nu ^\frac14 \left( 
            \nu \int _{\tau / 4} ^{5 \tau / 4} \int _{\Gammanuu} \lvert \nabla \unu \rvert ^2 \d x \d s
        \right) ^\frac12 \Vb \tau ^\frac14 S ^\frac12,
    \end{align*}
    where $\kappa = \lVert \nabla \bV \rVert _\infty \Ub ^{-1} + \lVert \partial _t \bV \rVert _\infty \Ub ^{-2}$ and $\Gammanuu$ is as in \cref{thm:main}.
\end{lemma}

The proof of this lemma relies on a dyadic decomposition of the boundary that is almost identical to \cite{vasseur2023,vasseur2024}. We give details of the proof in \cref{app:cz}. For now, we will invoke \cref{lem:boundary-work} as a black box and return to the proof of \cref{prop:boundary-work}.

\subsection{Case \texorpdfstring{$t \in [0, \Tone]$}{t in [0, T1]}}
\label{sub:first-epoch}

In this subsection, we show that if $t \le T _1$, then the integral 
\renewcommand{\Tone}{t}
\begin{align*}
    I _1 &:= \nu \int _0 ^{\Tone} \int _{\Gamma ^-} \partial _n \unu \cdot \bV \d \sigma \d s \to 0
\end{align*}
vanishes as $\nu \to 0$. To see this fact, we first note that 
\begin{align*}
    I _1 &= \nu \int _0 ^{\Tone} \int _\Omega \Delta \unu \cdot \bV \d \bx \d s + \nu \int _0 ^{\Tone} \int _\Omega \nabla \unu : \nabla \bV \d \bx \d s
    \\
    &= I_{1,1}+I_{1,2}.
\end{align*}
We observe that $I_{1,2}$ vanishes as $\nu \to 0$. Indeed, for $\nabla \unu \in L ^2(0,T;L^2(\Omega))$ and $\nabla \bV \in L ^\infty((0,T)\times \Omega)$:
\begin{align*}
    \left|I_{1,2}\right|
    \le \nu \lVert \nabla \unu \rVert _{L ^2 (0, T;L^2(\Omega))} \lVert \nabla \bV \rVert _{L ^\infty((0,T)\times\Omega)} (\Tone |\Omega|) ^\frac12 = \cO (\nu ^\frac32),
\end{align*}
where we applied \eqref{est:energy-inequality} to bound $\lVert \nabla \unu \rVert _{L ^2 ((0, T) \times \Omega)}$. On the other hand, for the first term, we use the equation for $\unu$:
\begin{align*}
    I_{1,1} &= \int _0 ^{\Tone} \int_{\Omega}\partial _t \unu \cdot \bV \d \bx \d s + \int _0 ^{\Tone} \int_{\Omega} \unu \cdot \nabla \unu \cdot \bV \d \bx \d s
    \\
    &=I_{1,1}^a+I_{1,1}^b
\end{align*}
Note that we invoked the fact that $\bV$ is divergence-free and tangent to the boundary, which implies that it is orthogonal to the pressure gradient $\nabla \pnu$ in $L ^2 (\Omega)$. The second term above is again vanishing as $\nu \to 0$ as
\begin{align*}
   I_{1,1}^b
    & \le \lVert \unu \rVert _{L ^\infty (0, T; L ^2 (\Omega))} \lVert \nabla \unu \rVert _{L ^2 (0, T; L ^2 (\Omega))} 
    \lVert \bV \rVert _{L ^\infty (0, T; L ^\infty (\Omega))} \Tone^{\frac12} = \cO (\nu ^\frac12).
\end{align*}
We next bound the first term $I _{1, 1} ^a$, which via an integration by parts becomes
\begin{align*}
    \int _0 ^{\Tone} \int_{\Omega}\partial _t \unu \cdot \bV \d \bx \d s = (\unu (\Tone), \bV (\Tone)) - (\unu (0), \bV (0)) - \int _0 ^{\Tone}\int_{\Omega} \unu \cdot \partial _t \bV \d \bx \d s.
\end{align*}
Here $(\cdot, \cdot)$ stands for $L ^2 (\Omega)$-inner product. Then this last term is vanishing since
\begin{align*}
    \left \lvert 
        \int _0 ^{\Tone} \int_{\Omega}\unu \cdot \partial _t \bV \d \bx \d s 
    \right \rvert 
    & \le \lVert \unu \rVert _{L ^\infty (0,T; L ^2 (\Omega))} \lVert \partial _t \bV \rVert _{L ^\infty (0, T; L ^\infty (\Omega))} \Tone = \cO (\nu ^2).
\end{align*}

It remains to prove 
\begin{align*}
    (\unu (\Tone), \bV (\Tone)) - (\unu (0), \bV (0)) \to 0 \qquad \text{as } \nu \to 0.
\end{align*}
Because $\{\unu\} _\nu$ is uniformly bounded in $L ^\infty (0, T; L ^2 (\Omega))$ and $\bV$ is strongly continuous in $L ^2 (\Omega)$, we have
\begin{align*}
    (\unu (\Tone), \bV (\Tone)) - (\unu (\Tone), \bV (0)) \to 0 \qquad \text{as } \nu \to 0.
\end{align*}
Therefore, it suffices to show
\begin{align*}
    (\unu (\Tone) - \unu (0), \bV (0)) \to 0 \qquad \text{as } \nu \to 0.
\end{align*}

From the equation satisfied by $\unu$, it is straightforward to see that $\{\partial _t \unu\} _\nu$ is uniformly bounded in $L ^{\frac43} (0, T; H ^{-10} (\Omega))$. Because $\{\unu\} _\nu$ is uniformly bounded in $L ^\infty (0, T; L ^2 (\Omega))$, we know that $\unu$ converges strongly in $C (0, T; H ^{-1} (\Omega))$ up to a subsequence by the Aubin--Lions Lemma. Therefore, for any $\boldsymbol W \in H ^1 _0 (\Omega)$, it holds that 
\begin{align*}
    (\unu (\Tone) - \unu (0), \boldsymbol W) \to 0 \qquad \text{as } \nu \to 0.
\end{align*}
Since $\{\unu\} _\nu$ is uniformly bounded in $L ^\infty (0, T; L ^2 (\Omega))$, we can approximate $\bV (0)$ by $H _0 ^1 (\Omega)$-functions in $L ^2 (\Omega)$, and the convergence holds for $\bV (0)$ as well. This completes the proof of $I _1 \to 0$.
\renewcommand{\Tone}{T _1}

\subsection{Case \texorpdfstring{$t \in [\Tone, \Ttwo]$}{t in [T1, T2]}}
\label{sub:second-epoch}

We choose $N \in \mathbb N$ such that 
\begin{align*}
    \frac45 T _1 \le T _1 ' := \left(\frac45\right) ^{N} t \le \Tone.
\end{align*}
With such a selection, we decompose
\begin{align*}
    \nu \int _{0} ^{t} \int _{\Gamma ^-} \partial _n \unu \cdot \bV \d \sigma = \nu \int _{0} ^{T _1'} \int _{\Gamma ^-} \partial _n \unu \cdot \bV \d \sigma + \nu \int _{T _1'} ^{t} \int _{\Gamma ^-} \partial _n \unu \cdot \bV \d \sigma = I _1' + I _2,
\end{align*}
and we show both integrals vanish as $\nu \to 0$. As $T _1' \le T _1$, $I _1'$ tends to zero as in \cref{sub:first-epoch}. To show the second integral vanishes, we partition 
\begin{align*}
    [T _1', t] = \bigcup _{k = 0} ^{N - 1} \left[
    \left(\frac54\right) ^k T _1', \left(\frac54\right) ^{k + 1} T _1'
    \right).
\end{align*}
We set $t _k = (5 / 4) ^k T _1'$. Using \cref{lem:boundary-work} with $\tau = t _k$, we conclude that
\begin{align*}
    \left\lvert 
        \nu \int _{t _k} ^{t _{k + 1}} \int _{\Gamma ^-} \partial _n \unu \cdot \bV \d \sigma \d s
    \right\rvert &\lesssim \kappa \nu ^2 \left(
        \int _{t _k} ^{t _{k + 1}} \int _{\Gamma ^-} \lvert \partial _n \unu \rvert ^\frac43 \d \sigma \d s
    \right) ^\frac34 t _k ^\frac14 S ^{\frac14} \\
    & \qquad + \left( 
        \nu \int _{t _k / 4} ^{t _{k + 1}} \int _{\Gammanuu} \lvert \nabla \unu \rvert ^2 \d x \d s
    \right) ^\frac23 \Vb t _k  ^\frac13 S ^\frac13 \\
    & \qquad + \nu ^\frac14 \left( 
        \nu \int _{t _k / 4} ^{t _{k + 1}} \int _{\Gammanuu} \lvert \nabla \unu \rvert ^2 \d x \d s
    \right) ^\frac12 \Vb t _k ^\frac14 S ^\frac12.
\end{align*}
It then follows from the trace estimate \eqref{est:trace} for $\partial _n \unu$ that
\begin{align*}
    \kappa \nu ^2 \left(
        \int _{t _k} ^{t _{k + 1}} \int _{\Gamma ^-} \lvert  \partial _n \unu \rvert ^\frac43 \d \sigma \d s
    \right) ^\frac34 t _k ^\frac14 S ^{\frac14} &\lesssim  
    \kappa \nu \left(
        \KE \nu ^{-\frac34} + \KE^{\frac12}|\Omega|^{\frac13} t _k ^{-\frac14}
    \right) t _k ^\frac14 S ^{\frac14} =\cO(\nu^{\frac12}).
\end{align*}
where we also invoked that $t _k \le T _2 = \nu \Ub ^{-2} / 4$. Moreover, we recall that the quantity $\nu \lVert \nabla \unu \rVert _{L ^2 ((0, T) \times \Omega)} ^2$ are uniformly bounded in $\nu$. Hence we have 
\begin{align*}
    \left( 
        \nu \int _{t _k / 4} ^{t _{k + 1}} \int _{\Gammanuu} \lvert \nabla \unu \rvert ^2 \d x \d s
    \right) ^\frac23 \Vb t _k  ^\frac13 S &\lesssim \KE^{\frac23}\Vb S\frac{\nu^{\frac13}}{\Ub^{\frac23}}= \cO (\nu ^\frac13), \\
    \nu ^\frac14 \left( 
        \nu \int _{t _k / 4} ^{t _{k + 1}} \int _{\Gammanuu} \lvert \nabla \unu \rvert ^2 \d x \d s
    \right) ^\frac12 \Vb t _k ^\frac14 S ^\frac12 &\lesssim \nu^{\frac14} \KE^{\frac12}\Vb\frac{\nu^\frac14}{\Ub^\frac12} S^{\frac12}= \cO (\nu ^\frac12).
\end{align*}
To summarize, we obtain
\begin{align*}
    \left\lvert 
        \nu \int _{t _k} ^{t _{k + 1}} \int _{\Gamma ^-} \partial _n \unu \cdot \bV \d \sigma \d s
    \right\rvert = \cO (\nu ^\frac13).
\end{align*}
Since $N = \cO (\log (T _2 / T _1)) = \cO (\log (T \Ub ^2 / \nu))$, $I _2$ can be bounded by 
\begin{align*}
    |I _2| \le \sum _{k = 0} ^{N - 1} \left\lvert 
        \nu \int _{t _k} ^{t _{k + 1}} \int _{\Gamma ^-} \partial _n \unu \cdot \bV \d \sigma \d s
    \right\rvert = \cO (\nu ^\frac13 \log(T \Ub ^2 / \nu)) \to 0 \qquad \text{as } \nu \to 0.
\end{align*}

\subsection{Case \texorpdfstring{$t \in [\Ttwo, T]$}{t in [T2, T]}}

We choose $M \in \mathbb N$ such that 
\begin{align*}
    \frac78 T _2 \le T _2' := t - \frac M8 T _2 \le T _2.
\end{align*}
With such a selection, we decompose
\begin{align*}
    \nu \int _{0} ^{t} \int _{\Gamma ^-} \partial _n \unu \cdot \bV \d \sigma = \nu \int _{0} ^{T _2'} \int _{\Gamma ^-} \partial _n \unu \cdot \bV \d \sigma + \nu \int _{T _2'} ^{t} \int _{\Gamma ^-} \partial _n \unu \cdot \bV \d \sigma = I _2' + I _3.
\end{align*}
$I _2' \to 0$ as $\nu \to 0$ by arguing as in \cref{sub:second-epoch}, so the main task is to estimate $I _3$. To find an upper bound, we set $t _0 = T _2 / 8$ and split $[T _2', t]$ evenly:
\begin{align*}
    [T _2', t] = \bigcup _{k = 0} ^{M - 1} [T _2' + k t _0, T _2' + (k + 1) t _0].
\end{align*}
We remark the partition used in \cref{sub:second-epoch} is a geometric sequence, whereas this partition is an arithmetic sequence. 

We denote $t _k = T _2' + k t _0$ for $k \in \mathbb Z$. Using \cref{lem:boundary-work} with $\tau = 4 t _0$, and shifting the initial time to $t _k - 4 t _0 = t _{k - 4}$, we have that
\begin{align*}
    \left\lvert 
        \nu \int _{t _k} ^{t _{k + 1}} \int _{\Gamma ^-} \partial _n \unu \cdot \bV \d \sigma \d s
    \right\rvert &\lesssim \kappa \nu ^2 \left(
        \int _{t _k} ^{t _{k + 1}} \int _{\Gamma ^-} \lvert \partial _n \unu \rvert ^\frac43 \d \sigma \d s
    \right) ^\frac34 t _0 ^\frac14 S ^{\frac14} \\
    & \qquad + \left( 
        \nu \int _{t _{k - 3}} ^{t _{k + 1}} \int _{\Gammanuu} \lvert \nabla \unu \rvert ^2 \d x \d s
    \right) ^\frac23 \Vb t _0  ^\frac13 S ^\frac13 \\
    & \qquad + \nu ^\frac14 \left( 
        \nu \int _{t _{k - 3}} ^{t _{k + 1}} \int _{\Gammanuu} \lvert \nabla \unu \rvert ^2 \d x \d s
    \right) ^\frac12 \Vb t _0 ^\frac14 S ^\frac12.
\end{align*}
First, we use H\"older's inequality to bound 
\begin{align*}
    I _{3, 1} &:= \sum _{k = 0} ^M \kappa \nu ^2 \left(
        \int _{t _k} ^{t _{k + 1}} \int _{\Gamma ^-} \lvert \partial _n \unu \rvert ^\frac43 \d \sigma \d s
    \right) ^\frac34 t _0 ^\frac14 S ^{\frac14} \\
    & \le \kappa \nu ^2 \left(
        \sum _{k = 0} ^M \int _{t _k} ^{t _{k + 1}} \int _{\Gamma ^-} \lvert \partial _n \unu \rvert ^\frac43 \d \sigma \d s
    \right) ^{\frac34} \left(
        \sum _{k = 0} ^M t _0 S
    \right) ^\frac14 \\
    &= \kappa \nu ^2 \lVert \partial _n \unu \rVert _{L ^\frac43 ((T _2', T) \times \Gamma ^-)} (T S) ^\frac14 = \cO (\nu ^\frac14 \lvert \log(\nu/(T\Ub^2)) \rvert ^\frac34) \to 0
\end{align*}
which converges to $0$ as $\nu \to 0$. Note that we applied the trace estimate \eqref{est:trace-2} for $\partial _n \unu$ to obtain the final relation.

For the second term, we use H\"older's inequality again to obtain
\begin{align*}
    I _{3, 2} &:= \sum _{k = 0} ^M \left( 
        \nu \int _{t _{k - 3}} ^{t _{k + 1}} \int _{\Gammanuu} \lvert \nabla \unu \rvert ^2 \d x \d s
    \right) ^\frac23 \Vb t _0  ^\frac13 S ^\frac13 \\
    & \le \Vb
    \left( \sum _{k = 0} ^M 
        \nu \int _{t _{k - 3}} ^{t _{k + 1}} \int _{\Gammanuu} \lvert \nabla \unu \rvert ^2 \d x \d s
    \right) ^\frac23 \left(
        \sum _{k = 0} ^M t _0 S
    \right) ^\frac13 \\
    & \lesssim  \Vb \left( 
        \nu \int _0 ^T \int _{\Gammanuu} \lvert \nabla \unu \rvert ^2 \d x \d s
    \right) ^\frac23 \left(
        T S
    \right) ^\frac13.
\end{align*}
Similarly, because $\nu\sim t _0 \Ub ^2$, we have 
\begin{align*}
    I _{3, 3} &:= \nu ^\frac14 \sum _{k = 0} ^M \left( 
        \nu \int _{t _{k - 3}} ^{t _{k + 1}} \int _{\Gammanuu} \lvert \nabla \unu \rvert ^2 \d x \d s
    \right) ^\frac12 \Vb t _0  ^\frac14 S ^\frac12 \\
    & \lesssim \Ub ^\frac12 \Vb \left( 
        \nu \int _0 ^T \int _{\Gammanuu} \lvert \nabla \unu \rvert ^2 \d x \d s
    \right) ^\frac12 \left(
        T S
    \right) ^\frac12.
\end{align*}
Combining these estimates yields
\begin{align*}
    |I _3| & \lesssim \Vb \left( 
        \nu \int _0 ^T \int _{\Gammanuu} \lvert \nabla \unu \rvert ^2 \d x \d s
    \right) ^\frac23 \left(
        T S
    \right) ^\frac13 \\
    & \qquad + \Ub ^\frac12 \Vb \left( 
        \nu \int _0 ^T \int _{\Gammanuu} \lvert \nabla \unu \rvert ^2 \d x \d s
    \right) ^\frac12 \left(
        T S
    \right) ^\frac12 + \cO (\nu ^\frac14 \lvert \log(\nu/(T\Ub^2)) \rvert ^\frac34).
\end{align*}
Together with the first two cases, we conclude the proof of \cref{prop:boundary-work}.

\section{Proof of main result}
\label{sec:main-proof}

We are ready to prove the main result.

\begin{proof}[Proof of \cref{thm:main}]

By applying \cref{prop:boundary-work} to equation \eqref{eqn:layer-separation-rate-2} in \cref{sec:boundary-layer}, we have the bound
\begin{align*}
    & \hspace{-2em} \lVert \unu - \uE \rVert _{L ^2 (\Omega)} ^2 (t) 
    + \nu \int _0 ^t \lVert \nabla \unu \rVert _{L ^2 (\Omega)} ^2 (s) \d s 
    + \frac12 \int _0 ^t \int _{\Gamma ^-} \lvert \bV \rvert ^2 U \normal \d \sigma \d s \\
    & \le 2 \int _0 ^t \lVert \nabla \uE (s) \rVert _{L ^\infty (\Omega)} \lVert \unu - \uE \rVert _{L ^2 (\Omega)} ^2 (s) \d s \\
    & \qquad + c \Vb \left( 
        \nu \int _0 ^t \int _{\Gammanuu} \lvert \nabla \unu \rvert ^2 \d x \d s
    \right) ^\frac23 \left(
        T S
    \right) ^\frac13 \\
    & \qquad + c \Ub ^\frac12 \Vb \left( 
        \nu \int _0 ^t \int _{\Gammanuu} \lvert \nabla \unu \rvert ^2 \d x \d s
    \right) ^\frac12 \left(
        T S
    \right) ^\frac12 + R _\nu ^1
\end{align*}
for some non-dimensional positive constant $c$ and a remainder term 
\begin{align*}
    R _\nu ^1 = \lVert \unu - \uE \rVert _{L ^2 (\Omega)} ^2 (0) + \nu \int _0 ^T \lVert \nabla \uE \rVert _{L ^2 (\Omega)} ^2 (s) \d s + R _\nu
\end{align*}
that tends to zero as $\nu \to 0$. We recall that $K \geq \exp (2 \int _0 ^T \lVert \nabla \uE (t) \rVert _{L ^\infty (\Omega)} \d t)$. 

Next, we can bound two terms on the right-hand side as follows:
\begin{align*}
    & c \Vb \left( 
        \nu \int _0 ^t \int _{\Gammanuu} \lvert \nabla \unu \rvert ^2 \d x \d s
    \right) ^\frac23 \left(
        S T
    \right) ^\frac13 \\
    & \qquad \le \frac\nu4 \int _0 ^t \lVert \nabla \unu \rVert _{L ^2 (\Omega)} ^2 (s) \d s + c S T \Vb ^3 \\
    & \qquad = \frac\nu4 \int _0 ^t \lVert \nabla \unu \rVert _{L ^2 (\Omega)} ^2 (s) \d s + c \beta S T \Ub \Vb ^2, \\
    & c \Ub ^\frac12 \Vb \left( 
        \nu \int _0 ^t \int _{\Gammanuu} \lvert \nabla \unu \rvert ^2 \d x \d s
    \right) ^\frac12 \left(
        S T
    \right) ^\frac12 \\
    & \qquad \le \frac\nu4 \int _0 ^T \lVert \nabla \unu \rVert _{L ^2 (\Omega)} ^2 (s) \d s + c S T \Ub \Vb ^2.
\end{align*}
Therefore,
\begin{align*}
    & \lVert \unu - \uE \rVert _{L ^2 (\Omega)} ^2 (t) 
    + \frac\nu2 \int _0 ^T \lVert \nabla \unu \rVert _{L ^2 (\Omega)} ^2 (s) \d s \\
    & \qquad \le 2 \int _0 ^t \lVert \nabla \uE (s) \rVert _{L ^\infty (\Omega)} \lVert \unu - \uE \rVert _{L ^2 (\Omega)} ^2 (s) \d s + c (1 + \beta) S T \Ub \Vb ^2 + R _\nu ^1.
\end{align*}
This estimate holds for all $t \in [0, T]$. By Gr\"onwall's inequality, we conclude that 
\begin{align*}
    \lVert \unu - \uE \rVert _{L ^2 (\Omega)} ^2 (T) 
    + \nu \int _0 ^T \lVert \nabla \unu \rVert _{L ^2 (\Omega)} ^2 (s) \d s \le c K \left(
        c (1 + \beta) S T \Ub \Vb ^2 + R _\nu ^1
    \right).
\end{align*}
However, the following bounds also hold:
\begin{align*}
    & c \Vb \left( 
        \nu \int _0 ^t \int _{\Gammanuu} \lvert \nabla \unu \rvert ^2 \d x \d t
    \right) ^\frac23 \left(
        S T
    \right) ^\frac13 \\
    & \qquad \le c \left(\frac{c K \Vb}{\gamma \Ub}\right) ^\frac12 \nu \int _0 ^t \lVert \nabla \unu \rVert _{L ^2 (\Gammanuu)} ^2 (s) \d s + \frac\gamma{4 c K} S T \Ub \Vb ^2 \\
    & \qquad \le c \left(\frac{4c K \beta}{\gamma}\right) ^\frac12 \nu \int _0 ^t \lVert \nabla \unu \rVert _{L ^2 (\Gammanuu)} ^2 (s) \d s + \frac\gamma{4 c K} S T \Ub \Vb ^2, \\
    & c \Ub ^\frac12 \Vb \left( 
        \nu \int _0 ^t \int _{\Gammanuu} \lvert \nabla \unu \rvert ^2 \d x \d t
    \right) ^\frac12 \left(
        S T
    \right) ^\frac12 \\
    & \qquad \le c\left(\frac { c K}\gamma\right) \nu \int _0 ^t \lVert \nabla \unu \rVert _{L ^2 (\Gammanuu)} ^2 (s) \d s + \frac\gamma{4 c K} S T \Ub \Vb ^2.
\end{align*}
Upon applying the second condition in \eqref{cond:smooth-lower}, it then follows from Gr\"onwall's inequality that
\begin{align*}
    & \lVert \unu - \uE \rVert _{L ^2 (\Omega)} ^2 (T) 
    + \frac\gamma2 S T \Ub \Vb ^2 \\
    & \qquad \le \frac{c K (1 + \sqrt{\beta \gamma})}{\gamma}  \left(
        \nu \int _0 ^T \int _{\Gammanuu} \lvert \nabla \unu \rvert ^2 \d x \d t + R _\nu ^1
    \right).
\end{align*}

By combining the estimates above and sending $\nu \to 0$, we obtain the following \textit{upper and lower bound} on the energy dissipation rate:
\begin{align*}
    \frac{\gamma ^2}{c K \left(1 + \sqrt{\beta \gamma} \right)} S T \Ub \Vb ^2 &\le \liminf _{\nu \to 0} \int _0 ^T \int _{\Gammanuu} \nu \lvert \nabla \unu \rvert ^2 \d x \d t \\
    & \le \limsup _{\nu \to 0} \int _0 ^T \int _{\Omega} \nu \lvert \nabla \unu \rvert ^2 \d x \d t \le c K (1 + \beta) S T \Ub \Vb ^2,
\end{align*}
and the following {\em upper bound} on the layer separation rate:
\begin{align*}
    \limsup _{\nu \to 0} \lVert \unu - \uE \rVert _{L ^2 (\Omega)} ^2 (T) \le c K (1 + \beta) S T \Ub \Vb ^2.
\end{align*}
This completes the proof of \cref{thm:main}.
\end{proof}

\bibliographystyle{abbrv}
\bibliography{reference.bib}

\begin{appendix}
\section{Proof of \texorpdfstring{\cref{lem:boundary-work}}{Lemma 4.2}}
\label{app:cz}

The proof of \cref{lem:boundary-work} relies on a decomposition of the boundary that we now describe in more detail. The main idea is to partition the boundary into a series of boxes, such that if we rescale each box to a unit box, the energy dissipation is of order $1$, so that \cref{lem:local} applies.

In what follows, given two measurable sets $A$ and $B$,  $A \approx B$ means the two sets differ by a measure-zero set.

\begin{lemma}
    \label{lem:l32weaknew}
    Under the hypotheses of \cref{lem:boundary-work}, let $\varepsilon = \sqrt{\nu \tau} < \nu / (2 \Ub)$. The set $(\tau, 5\tau / 4) \times \Gamma ^-$ can be partitioned into
    \begin{align*}
        \left(\tau, 5\tau / 4\right) \times \Gamma ^- \approx \bigcup _{i \in \mathbb N} \bar Q _{(i)} = \bigcup _{i \in \mathbb N} \bar Q _{\varepsilon _i} (t _i, \bx _i),
    \end{align*}
    where $\varepsilon _i \le \varepsilon / 2$, in such a way that the piecewise-constant function $\tomega ^\nu: (\tau, 5\tau / 4) \times \Gamma ^- \to \mathbb R ^d$ defined by
    \begin{align*}
        \tomega ^\nu \big\vert _{\bar Q _{(i)}} = \fint _{\bar Q _{(i)}} \partial _n \unu \d \sigma \d t
    \end{align*}
    satisfies :
    \begin{align*}
        \left\lVert \nu \tomega ^\nu \mathbf1 _{\{|\nu \tomega ^\nu| \ge c \nu \eta \tau ^{-1}\}} \right\rVert _{L ^{\frac32, \infty} ((\tau, 5\tau / 4) \times \Gamma ^-)} ^\frac32 \le {c \eta ^{-\frac12}}\int _{\tau/4} ^{5\tau / 4} \int _{\Gammanuu} \nu \lvert \nabla \unu \rvert ^2 \d x \d t,
    \end{align*}
    for some non-dimensional universal constant $c > 0$.
    In other words, for any $M > c \eta \nu \tau ^{-1}$, it holds that 
    \begin{align*}
        \left\lvert\left\lbrace
            (t, \bx) \in \left(\tau, 5\tau / 4\right) \times \Gamma ^-: |\nu \tomega ^\nu (t, \bx)| > M
        \right\rbrace\right\rvert < \frac {c\eta ^{-\frac12}}{M ^\frac32} \int _{\tau/4} ^{5\tau / 4} \int _{\Gammanuu} \nu \lvert \nabla \unu \rvert ^2 \d x \d t.
    \end{align*}
\end{lemma}

\begin{proof}
    We construct the dyadic decomposition using the same strategy as in \cite{vasseur2023}, so below we only sketch the main idea. For simplicity, we assume the boundary is flat, while the general curved boundary can be handled identically as in \cite{vasseur2024}. We may also assume $\Gamma ^-$ can be decomposed into a finite disjoint union of boxes $\bar B _{\varepsilon / 2} (\bx _i)$ of size $\varepsilon / 2$. Thus $(\tau, 5\tau / 4) \times \Gamma ^-$ can be partitioned into the union of $\bar Q _{(i)} = (t _i - \tau _i, t _i) \times \bar B _{\varepsilon _i} (\bx _i)$ with $\varepsilon _i = \varepsilon / 2$, $t _i = 5 \tau / 4$ and $\tau _i = \tau / 4$. For each $\bar Q _{(i)}$, we divide it dyadically into $2 ^d$ smaller pieces \textit{unless} 
    \begin{align}
        \label{eqn:dyadic-stop}
        \left( 
            \fint _{Q _{2 \varepsilon _i} ^+ (t _i, x _i)} \lvert \nabla \unu \rvert ^2 \d x \d t
        \right) ^\frac12 \le \eta \nu \varepsilon _i ^{-2}.
    \end{align}
    This partitioning process may stop after finitely many steps, or continue indefinitely, leaving a zero-measure remainder. Eventually we obtain a partition such that \eqref{eqn:dyadic-stop} is satisfied for each $\bar Q _{(i)}$ (up to re-indexing). Moreover, for each $\varepsilon _i < \varepsilon / 2$ the following holds: 
    \begin{align*}
        \left( 
            \fint _{Q _{2 \varepsilon _i} ^+ (t _i, x _i)} \lvert \nabla \unu \rvert ^2 \d x \d t
        \right) ^\frac12 > \frac1{16} \eta \nu \varepsilon _i ^{-2},
    \end{align*}
    because the partition did not stop at the previous step.

    From \cref{lem:local}, we know that in each $\bar Q _{(i)}$, 
    \begin{align*}
        \left\lvert 
            \tomega ^\nu \big\vert _{\bar Q _{(i)}}
        \right\rvert \le c \eta \nu \varepsilon _i ^{-2}.
    \end{align*}
    If $\varepsilon _i = \varepsilon / 2$, then $\tau _i = \tau / 4$ and
    \begin{align*}
        \left\lvert 
            \nu \tomega \big\vert _{\bar Q _{(i)}}
        \right\rvert \le c \eta \nu ^2 \varepsilon ^{-2} = c \eta \nu \tau ^{-1} < M.
    \end{align*}
    Let $\mathcal M$ be the maximal operator in $\mathbb R \times \mathbb R ^d$. For $(t, x) \in (0, T) \times \Omega$, denote 
    \begin{align*}
        m (t, x) = \mathcal M \left(|\nabla \unu| ^2 \mathbf1 _{(\tau / 4, 5 \tau / 4) \times \Gammanuu}\right).
    \end{align*}
    Then for every $\varepsilon _i < \varepsilon / 2$ and every $(t, x) \in Q _{(i)} ^+$, it holds that 
    \begin{align*}
        \eta \nu \varepsilon _i ^{-2} < 16 \left( 
            \fint _{Q _{2 \varepsilon _i} ^+ (t _i, x _i)} \lvert \nabla \unu \rvert ^2 \d x \d t
        \right) ^\frac12 \le c m ^{\frac12} (t, x).
    \end{align*}
    Denote 
    \begin{align*}
        \mathcal I _k = \left\lbrace
            i \in \mathbb N: 2 ^{-k-1} \varepsilon \le \varepsilon _i < 2 ^{-k} \varepsilon
        \right\rbrace, \qquad k \ge 0.
    \end{align*}
    Then for every $i \in \mathcal I _k$, it holds that  
    \begin{align*}
        \left\lvert 
            \nu \tomega ^\nu \big\vert _{\bar Q _{(i)}}
        \right\rvert \le c \eta \nu ^2 \varepsilon _i ^{-2} \le c \eta \nu ^2 \varepsilon ^{-2} 2 ^{2k} =: M _k.
    \end{align*}
    Moreover, the last two quantities are comparable by a factor of $2 ^2$, so 
    \begin{align*}
        M _k ^\frac32 \le 8 c \eta ^\frac32 \nu ^3 \varepsilon _i ^{-3} = c \eta ^{-\frac12} \varepsilon _i \nu \cdot (\eta \nu \varepsilon _i ^{-2}) ^2 \le c \eta ^{-\frac12} \varepsilon _i \nu m (t, x)
    \end{align*}
    for every $(t, x) \in Q _{(i)} ^+$, so that
    \begin{align*}
        \nu m \big\vert _{Q _{(i)} ^+} \ge \frac{1}{c} 2 ^k \varepsilon ^{-1} \eta ^\frac12 M _k ^\frac32.
    \end{align*}
    Therefore,  the total measure of these cubes is bounded by 
    \begin{align*}
        \sum _{i \in \mathcal I _k} |Q _{(i)} ^+| &\le \left\lvert 
            \left\lbrace\nu m \ge \frac{1}{c} 2 ^k \varepsilon ^{-1} \eta ^\frac12 M _k ^\frac32
            \right\rbrace
        \right\rvert \\
        &\le c 2 ^{-k} \varepsilon \eta ^{-\frac12} M _k ^{-\frac32} \left\lVert \nu m \right\rVert _{L ^{1, \infty} (\mathbb R \times \mathbb R ^d)} \\
        & \le c 2 ^{-k} \varepsilon \eta ^{-\frac12} M _k ^{-\frac32} \left\lVert \nu |\nabla \unu| ^2 \mathbf1 _{(\tau / 4, 5 \tau / 4) \times \Gammanuu} \right\rVert _{L ^1 (\mathbb{R} \times \mathbb{R} ^d)} \\
        &= c 2 ^{-k} \varepsilon \eta ^{-\frac12} M _k ^{-\frac32} \int _{\tau / 4} ^{5 \tau / 4} \int _{\Gammanuu} \nu \lvert \nabla \unu \rvert ^2 \d x \d t.
    \end{align*}
    Note that $|Q _{(i)} ^+| = \varepsilon _i |\bar Q _{(i)}| \ge 2 ^{-k-1} \varepsilon |\bar Q _{(i)}|$. Hence
    \begin{align*}
        \left\lvert\left\lbrace
            |\nu \tomega ^\nu| > M
        \right\rbrace\right\rvert 
        &\le \sum _{i \in \mathbb N} \left\lbrace
            |\bar Q _{(i)}| : i \in \mathcal I _k, M _k > M
        \right\rbrace \\
        &= \sum _{k \in \mathbb N, M _k \ge M} \sum _{i \in \mathcal I _k} |\bar Q _{(i)}| \\
        &= \sum _{k \in \mathbb N, M _k \ge M} 2 ^{k + 1} \varepsilon ^{-1} \sum _{i \in \mathcal I _k} |Q _{(i)} ^+| \\
        &\le c \sum _{k \in \mathbb N, M _k \ge M} \eta ^{-\frac12} M _k ^{-\frac32} \int _{\tau / 4} ^{5 \tau / 4} \int _{\Gammanuu} \nu \lvert \nabla \unu \rvert ^2 \d x \d t \\
        &= c \eta ^{-\frac12} M ^{-\frac32} \int _{\tau / 4} ^{5 \tau / 4} \int _{\Gammanuu} \nu \lvert \nabla \unu \rvert ^2 \d x \d t.
    \end{align*}
    This completes the proof of the lemma.
\end{proof}

We are now ready to prove \cref{lem:boundary-work} using this partition.

\begin{proof}[Proof of \cref{lem:boundary-work}]
First, using the partition in \cref{lem:l32weaknew}, we decompose 
    \begin{align*}
        \nu \int _{\tau} ^{5 \tau / 4} \int _{\Gamma ^-} \partial _n \unu \cdot \bV \d \sigma \d s 
        &= \sum _{i \in \mathbb N} \nu \int _{\bar Q _{(i)}} \partial _n \unu \cdot \bV \d \sigma \d s \\
        &= \sum _{i \in \mathbb N} \nu \int _{\bar Q _{(i)}} \tomega ^\nu \cdot \bV \d \sigma \d s \\
        & \qquad + \sum _{i \in \mathbb N} \nu \int _{\bar Q _{(i)}} \partial _n \unu \cdot \left(\bV - \left(\fint _{\bar Q _{(i)}} \bV \right) \right) \d \sigma \d s. 
    \end{align*}
Here we used the fact that taking the average defines a projection operator, which is self-adjoint. First, we study the second term. Take $(t, \bx) \in \bar Q _{(i)}$, then 
    \begin{align*}
        \left\lvert \bV (t, \bx) - \left(\fint _{\bar Q _{(i)}} \bV \right) \right\rvert
        &\le \left\lVert \nabla \bV \right\rVert _\infty \varepsilon _i + \left\lVert \partial _t \bV \right\rVert _\infty \varepsilon _i ^2 \nu ^{-1} \\
        &\le \left\lVert \nabla \bV \right\rVert _\infty \varepsilon + \left\lVert \partial _t \bV \right\rVert _\infty \varepsilon ^2 \nu ^{-1} \\
        &\le \nu \left(
            \lVert \nabla \bV \rVert _\infty \Ub ^{-1} + \lVert \partial _t \bV \rVert _\infty \Ub ^{-2}
        \right) = \kappa \nu. 
    \end{align*}
    Therefore
    \begin{align*}
        \sum _{i \in \mathbb N} \nu \int _{\bar Q _{(i)}} \partial _n \unu \cdot \left(\bV - \left(\fint _{\bar Q _{(i)}} \bV \right) \right) \d \sigma \d s \lesssim \kappa \nu ^2 
        \lVert \partial _n \unu \rVert _{L ^\frac43 ((\tau, 5 \tau / 4) \times \Gamma ^-)} (\tau S) ^{\frac14}.
    \end{align*}

    Next, we study the first term. Using \cref{lem:l32weaknew}, we have
    \begin{align*}
        & \nu \int _{\tau} ^{5 \tau / 4} \int _{\Gamma ^-} \tomega ^\nu \cdot \bV \d \sigma \d s \\
            & \qquad = \int _{\tau} ^{5 \tau / 4} \int _{\Gamma ^-} \nu \tomega ^\nu \cdot \bV \mathbf1 _{\{|\nu \tomega ^\nu| \ge c \nu \eta \tau ^{-1}\}} \d \sigma \d s + \int _{\tau} ^{5 \tau / 4} \int _{\Gamma ^-} c \nu \eta \tau ^{-1} \cdot |\bV| \d \sigma \d s \\
        & \qquad \le c \left( 
            \eta ^{-\frac12} \int _{\tau / 4} ^{5 \tau / 4} \int _{\Gammanuu} \nu \lvert \nabla \unu \rvert ^2 \d x \d t
        \right) ^\frac23 \lVert \bV \rVert _{L ^{3, 1} ((\tau, 5 \tau / 4) \times \Gamma ^-)} + c \nu \eta \Vb S \\
        & \qquad \le c \left( 
            \eta ^{-\frac12} \int _{\tau / 4} ^{5 \tau / 4} \int _{\Gammanuu} \nu \lvert \nabla \unu \rvert ^2 \d x \d t
        \right) ^\frac23 \Vb (\tau S) ^\frac13 + c \nu \eta \Vb S.
    \end{align*}
    This estimate holds for any $\eta \le 1$. Optimizing $\eta \in (0, 1]$ yields 
    \begin{align*}
        \nu \int _{\tau} ^{5 \tau / 4} \int _{\Gamma ^-} \tomega ^\nu \cdot \bV \d \sigma \d s & \lesssim \left( 
            \int _{\tau / 4} ^{5 \tau / 4} \int _{\Gammanuu} \nu \lvert \nabla \unu \rvert ^2 \d x \d t
        \right) ^\frac23 \Vb \tau  ^\frac13 S ^\frac13 \\
        & \qquad + \nu ^\frac14 \left( 
            \int _{\tau / 4} ^{5 \tau / 4} \int _{\Gammanuu} \nu \lvert \nabla \unu \rvert ^2 \d x \d t
        \right) ^\frac12 \Vb \tau ^\frac14 S ^\frac12.
    \end{align*}
    This completes the proof of the lemma.
\end{proof}

\end{appendix}

\end{document}